\documentclass[a4paper, 12pt]{amsart}
\usepackage{amsmath,amsfonts,amssymb,mathrsfs}
\usepackage{latexsym, amscd, epsfig}
\usepackage[utf8]{inputenc}
\usepackage{color}
\usepackage{amsmath}
\usepackage{amssymb}
\usepackage{amsthm}
\usepackage{graphicx}
\usepackage{esint}
\usepackage[colorlinks=true,linkcolor=blue]{hyperref}
\usepackage{verbatim}
\usepackage{accents}
\newcommand{\ubar}[1]{\underaccent{\bar}{#1}}

\theoremstyle{plain}
\newtheorem{theorem}{Theorem}
\newtheorem{proposition}{Proposition}[section]
\newtheorem{lemma}[proposition]{Lemma}

\newtheorem{remark}[proposition]{Remark}

\theoremstyle{problem}
\newtheorem {problem}{Problem}

\numberwithin{equation}{section}
\def\Xint#1{\mathchoice
	{\XXint\displaystyle\textstyle{#1}}
	{\XXint\textstyle\scriptstyle{#1}}
	{\XXint\scriptstyle\scriptscriptstyle{#1}}
	{\XXint\scriptscriptstyle\scriptscriptstyle{#1}}
	\!\int}
\def\XXint#1#2#3{{\setbox0=\hbox{$#1{#2#3}{\int}$ }
		\vcenter{\hbox{$#2#3$ }}\kern-.6\wd0}}

\def\dashint{\Xint-}

\def\MH{{\mu,R}}
\def\a{\alpha}
\def\b{\beta}
\def\c{\cdot}
\def\d{\delta}

\def\g{\gamma}
\def\G{\Gamma}
\def\k{\kappa}
\def\K{\mathcal{K}}
\def\ld{\lambda}
\def\Ld{\Lambda}
\def\n{\nabla}
\def\o{\omega}
\def\O{\Omega}
\def\ov{\overline}
\def\p{\partial}
\def\q{\quad}
\def\r{\rho}

\def\ul{\ubar}
\def\v{\epsilon}
\def\B{\mathcal{B}}

\def\MO{\mathcal{O}}
\def\mn{\mathbf{n}}
\def\mp{\mathbf{p}}
\def\MG{\mathcal{G}}
\def\M{\mathfrak{M}}
\def\N{\mathcal{N}}
\def\R{\mathbb{R}}

\def\S{\mathcal{S}}
\def\t{\mathfrak{t}}
\def\u{\mathbf{u}}
\def\h{\mathfrak{h}}

\def\ur{\Upsilon}
\def\mcG{{\mathcal G}}

\def\bp{{\mathbf p}}

\usepackage{xspace}

\begin{document}
\title[two-dimensional jet flows for full Euler system]{\bf Two-dimensional jet flows for compressible full Euler system with general vorticity}
\author{Yan Li}
\address{School of Mathematical Sciences, Shanghai Jiao Tong University, 800 Dongchuan Road, Shanghai, 200240, China}
\email{liyanly@sjtu.edu.cn}

\begin{abstract}
	In this paper, we consider the well-posedness theory of two-dimensional compressible subsonic jet flows for steady full Euler system with general vorticity. Inspired by the analysis in \cite{LSTX2023}, we show that the stream function formulation for such system admits a variational structure. Then the existence and uniqueness of a  smooth subsonic jet flow can be established by the variational method developed by Alt, Caffarelli and Friedman. Furthermore, the far fields behavior of the flow and the existence of a critical upstream pressure are also obtained. 
\end{abstract}
		
\keywords{free boundary, full Euler system, jet, subsonic flows, vorticity.}
\subjclass[2010]{
	35Q31, 35R35, 35J20, 35J70, 35M32, 76N10}

\thanks{Updated on \today}
		
\maketitle
		
\section{Introduction and main results}
Steady compressible subsonic flows in infinitely long nozzles have been studied extensively. In \cite{MR0096477} Bers first asserted that there exists a unique subsonic irrotational flow in an infinitely long nozzle if the mass flux is less than a critical value. The rigorous mathematical proof of this assertion was given in \cite{MR2375709} for two-dimensional flows, in \cite{MR2644144} for three-dimensional axisymmetric flows, and in \cite{MR2824469} for arbitrary multi-dimensional flows. As for rotational flows (with non-zero vorticity), the well-posedness and critical mass flux of two-dimensional subsonic flows were obtained in \cite{MR2607929} by using stream function formulations. This method also works for non-isentropic flows (\cite{MR3023398,MR3017027,MR3914482}), three-dimensional axisymmetric flows (\cite{MR3017027,MR3537905}), flows with large vorticity (\cite{MR3196988,MR3914482}), flows in periodic nozzles (\cite{MR2879733}) and non-smooth nozzles (\cite{MR3283559}) and so on.

Compared with flows in fixed nozzles, jet flows are usually more difficult to analyze, since free boundaries appear when the flows  issue from the nozzles. Early researches on such free boundary problems relied on hodograph tranformation and comformal mappings, which were restricted to two-dimensional flows and special nozzles, see \cite{MR0088230,MR0119655}. Alt, Caffarelli and Friedman developed a new method to solve two-dimensional and three-dimensional axisymmetric jet problems for irrotational
flows (\cite{MR637494,MR682265,MR772122}) by investigating the associated Lagrangian functionals (\cite{MR618549,MR752578}). This variational approach turns out to be very powerful to deal with different kinds of free boundary problems, cf. \cite{MR679313,MR3670047,MR3814594,MR4127792,DW_vorticity,DT_finite,DH_two,CD2020} and references therein. More results about subsonic-sonic jet flows and jet flows in bounded domains can be found in \cite{MR3048597,MR3934109} and references therein.

Based on the variational method mentioned above, the well-posedness theory of two-dimensional isentropic subsonic jet flows with general incoming flows  was  established in \cite{LSTX2023}. It was shown that the jet problem for isentropic flows enjoys a variational structure even when the vorticity is non-zero, so that the framework developed by Alt, Caffarelli and Friedman can be applied to obtain the existence and uniqueness of smooth solutions to the subsonic jet problem. Furthermore, a critical mass flux of the flows was obtained in the sense that as long as the incoming mass flux is larger than the critical value, the well-posedness theory holds true. It is natural to ask whether this variational method can be adapted to study the non-isentropic jet flows. In this paper, we give a positive answer and establish the well-posedness of two-dimensional non-isentropic subsonic jet flows with general vorticity.

Two-dimensional steady compressible jet flows satisfy the following full Euler system
\begin{eqnarray}\label{Euler}
	\left\{ \begin{array}{llll}
		\displaystyle  \n\cdot(\r \mathbf{u})=0,\\
		\displaystyle \n \cdot(\r \mathbf{u}\otimes \mathbf{u})+\n P=0,\\
		\displaystyle \n\cdot(\r \u E+P\u)=0
	\end{array}
	\right.
\end{eqnarray}
in the flow region,
where $\u=(u_1,u_2)$ denotes the velocity, $\r$ is the density, $P$ is the pressure, and $E$ is the total energy of the flow. For the polytropic gas, one has
$$E=\frac{|\u|^2}{2}+\frac{P}{(\g-1)\r},$$
where the constant $\gamma>1$ is called the adiabatic exponent. The local sound speed and the Mach number of the flow are defined as
\begin{equation}\label{C and M}
	c=\sqrt{\frac{\g P}{\r}} \quad {\rm and} \quad M=\frac{|\u|}{c},
\end{equation}
respectively. The flow is called subsonic if $M<1$, sonic if $M=1$, and supersonic if $M>1$, respectively.

Now consider a nozzle in $\R^2$ bounded by two solid boundaries, which is symmetric about the $x_1$-axis. Denote
\begin{equation}\label{nozzle}
	\N_0:=\{(x_1,0): x_1\in \R\} \q{\rm and} \q \N_1:=\{(x_1,x_2): x_1=\Theta(x_2)\}
\end{equation}
as the symmetry axis and the solid upper boundary of the nozzle, respectively. Here $\Theta\in C^{1,\bar\a}([1,\bar{H}])$ for some given $\bar{H}>1$ and $\bar\a\in(0,1)$, moreover, it satisfies
\begin{equation}\label{nozzleb}
	\Theta(1)=0 \q {\rm and} \q \lim_{x_2\to\bar{H}-}\Theta(x_2)=-\infty,
\end{equation}
i.e., the outlet of the nozzle is at $A:=(0,1)$ and the nozzle is asymptotically horizontal with height $\bar{H}$ at upstream $x_1\to-\infty$. In order to study the uniqueness of solution, we further assume that there exists an $h_*\in[1,\bar H)$ such that the nozzle is monotone decreasing for $x_2\in(h_*,\bar H)$, i.e.,
\begin{equation}\label{nozzle3}
\Theta'(x_2)\leq 0 \quad\text{for } x_2\in (h_*,\bar H).
\end{equation}

The main goal of this paper is to study the following jet problem.

\begin{problem}\label{problem1}
Given a density $\bar \rho=\bar \rho(x_2)$,
a horizontal velocity $\bar u=\bar u(x_2)$ and a constant pressure $\bar P>0$ of the flow at the upstream, find $(\r, \u, P)$, the free boundary $\G$, and the  outer pressure $\ul{P}$ which we assume to be a positive constant, such that the following statements hold:
\begin{enumerate}
  \item[\rm (i)] The free boundary $\G$ joins the outlet of the nozzle as a continuous differentiable curve and tends asymptotically horizontal at downstream as $x_1\to\infty$.

  \item[\rm (ii)] The solution $(\r,\u,P)$ solves the Euler system \eqref{Euler} in the flow region $\mathcal{O}$ bounded by $\N_0$, $\N_1$ and $\G$. It
  takes the incoming data at the upstream, i.e.,
  \begin{equation}\label{upstream condition}
  \rho(x_1,x_2)\to\bar{\rho}(x_2),\q u_1(x_1,x_2)\to\bar{u}(x_2),\q
  P(x_1,x_2)\to\bar{P},
  \q \text{as } x_1\to-\infty.
  \end{equation}
  Moreover, it satisfies the boundary conditions
  \begin{equation}\label{boundary condition}
  P=\ul{P} \  \text{on } \G \q\text{and} \q \u\c\mn=0 \  \text{on } \N_1\cup \G,
  \end{equation}
  where $\mn$ is the unit normal along $\N_1\cup\G$.
\end{enumerate}
\end{problem}

The main results in this paper can be stated as follows.

\begin{theorem}\label{result}
Let the nozzle boundary $\N_1$ defined in \eqref{nozzle} satisfy \eqref{nozzleb}-\eqref{nozzle3} and $\bar P>0$ be a constant. Given a density $\bar\rho\in C^{1,1}([0,\bar{H}])$ and a horizontal velocity $\bar{u}\in C^{1,1}([0,\bar{H}])$ of the flow at the upstream. Suppose
\begin{equation}\label{thm_condition1}\begin{split}
&\bar\rho_*:=\inf_{x_2\in[0,\bar H]} \bar \rho(x_2)>0,
\q \bar{\rho}'(0)=0, \q \bar{\rho}'(\bar H)=0,\\
&\bar u_*:=\inf_{x_2\in[0,\bar H]}\bar u(x_2)>0,
\q \bar{u}'(0)=0,\q \bar{u}'(\bar{H})\geq0.
\end{split}
\end{equation}
There exist constants $\bar\k,\,\bar P_c>0$ depending on $\bar \rho_*$, $\bar u$, $\gamma$ and the nozzle, such that the following statements hold:
\begin{enumerate}
  \item[\rm (i)] (Existence and properties of solutions) For any pressure $\bar P>\bar P_c$ at the upstream, if
   \begin{equation}\label{thm_condition2}
  	\bar P\|\bar\rho'\|_{C^{0,1}([0,\bar H])}<\bar\k,
  \end{equation}
  then there exist functions  $\rho,\mathbf{u},P\in C^{1,\alpha}(\mathcal{O})\cap C^0(\overline{\mathcal{O}})$ (for any $\alpha\in (0,1)$) where $\mathcal O$ is the flow region, the free boundary $\Gamma$, and the outer pressure $\ul{P}$ such that $(\r,\u,P,\G,\ubar P)$ solves Problem {\ref{problem1}}. Furthermore, the following properties hold:
  \begin{enumerate} 
  	\item  (Smooth fit) The free boundary $\G$ joins the outlet of the nozzle as a $C^1$ curve.
  	
  	\item  The free boundary $\Gamma$ is given by a graph $x_1=\Upsilon(x_2)$, $x_2\in (\ubar H, 1]$, where $\Upsilon$ is a $C^{2,\alpha}$ function, $\ubar H\in(0,1)$, and $\lim_{x_2\rightarrow \ubar H+} \Upsilon(x_2)=\infty$. For $x_1$ sufficiently large, the free boundary  can also be written as $x_2=f(x_1)$, where $f$ is a  $C^{2,\alpha}$ function and satisfies
  	$$\lim_{x_1\rightarrow \infty}f(x_1)=\ubar H
  	\quad\text{and}\quad  \lim_{x_1\rightarrow \infty}f'(x_1)=0.$$
  	
  	\item The flow is globally uniformly subsonic and has negative vertical velocity in the flow region $\MO$, i.e.,
  	\begin{equation*}
  		\sup_{\overline{\MO}}\frac{|\u|^2}{c^2}<1 \q\text{and}\q u_2<0 \ \text{in }\MO.
  	\end{equation*}
 
  	\item (Upstream and downstream asymptotics) The flow satisfies the asymptotic behavior
  	\begin{equation}\label{asymptotic rup}
  		\begin{split}
  			\|(\r,u_1,u_2,P)(x_1,\c)-(\bar{\r}(\c),\bar{u}(\c),0,\bar{P})\|_{C^{1,\alpha}([0,\bar{H}))}\to0 \q\text{as } x_1\to-\infty
  	\end{split}\end{equation}
  	and
  	\begin{equation}\label{asymptotic downstream}
  		\begin{split}
  			\|(\r,u_1,u_2,P)(x_1,\c)-(\ul{\r}(\c),\ul{u}(\c),0,\ul{P})\|_{C^{1,\alpha}([0,\ul{H}))}\to0 \q\text{as } x_1\to\infty,
  	\end{split} \end{equation}
  	where $(\ul{\r},\ul{u})\in (C^{1,\alpha}([0,\ul{H}]))^2$ are positive functions. Moreover, the downstream asymptotics $\ul{\r}$, $\ul{u}$, $\ul{P}$ and $\ul{H}$ are uniquely determined by $\bar \rho$, $\bar{u}$, $\bar{P}$, $\gamma$, and $\bar H$.
  \end{enumerate}

  \item[\rm (ii)] (Uniqueness) The Euler flow which satisfies all properties in part (i) is unique.

  \item[\rm (iii)] (Critical pressure) $\bar P_c$ is the critical pressure at the upstream for the existence of subsonic jet flows in the following sense: either
  \begin{equation*}
  \sup_{\overline{\MO}}\frac{|\u|^2}{c^2}\to1 \q\text{as } \bar P\to \bar P_c+,
  \end{equation*}
  or there is no $\sigma>0$ such that for all $\bar P\in(\bar P_c-\sigma, \bar P_c)$, there are Euler flows satisfying all properties in part (i) and
  \begin{equation*}
  \sup_{\bar P\in(\bar P_c-\sigma,\bar P_c)}\sup_{\overline{\MO}}\frac{|\u|^2}{c^2}<1.
  \end{equation*}
\end{enumerate}
\end{theorem}

\begin{remark}
The condition \eqref{nozzle3} of the nozzle boundary $\N_1$ at far field is only used for the uniqueness of the solution and the existence of the critical pressure.
\end{remark}

\begin{remark}
 The conditions of $\bar u$ in Theorem \ref{result} are the same as that in \cite[Theorem 1]{LSTX2023}, where the density $\bar\rho$ is a positive constant.
\end{remark}

\begin{remark}\label{rmk_p_constant}
		Suppose the flow satisfies the asymptotic behavior with high order compatibility at the downstream, i.e.,
	\begin{align*}
		(\rho(x_1,x_2),u_1(x_1,x_2),u_2(x_1,x_2),P(x_1,x_2))\to(\ubar\rho(x_2),\ubar u_1(x_2),\ubar u_2(x_2),\ubar P(x_2))
	\end{align*}
	and
	\begin{equation*}
		\n(\rho(x_1,x_2),u_1(x_1,x_2),u_2(x_1,x_2),P(x_1,x_2))\to\n(\ubar\rho(x_2),\ubar u_1(x_2),\ubar u_2(x_2),\ubar P(x_2))
	\end{equation*}
	for some smooth functions $(\ubar\rho,\ubar u_1,\ubar u_2,\ubar P)$ as $x_1\to\infty$. It follows from the Euler system \eqref{Euler} that
	\begin{equation}\label{eq_0}
		(\ubar\rho\ubar u_2)'=0,\q (\ubar\r \ubar u_1\ubar u_2)'=0,\q (\ubar \r \ubar u_2^2)'+\ubar P'=0,\q  \left(\ubar\r \ubar u_2\left(\frac{\ubar u_1^2+\ubar u_2^2}{2}+\frac{\g \ubar P}{(\g-1)\ubar\r}\right)\right)'=0.
	\end{equation}
	The first equation in \eqref{eq_0} shows that $\ubar\rho \ubar u_2\equiv C$ for some constant $C$. Since $\ubar\r \ubar u_2=0$ on the $x_1$-axis (since the flow is symmetric about the $x_1$-axis) implies $C=0$ and since the solution is supposed to satisfy $\ubar\rho>0$, one has $\ubar u_2=0$. Then it follows from the third equation in \eqref{eq_0} that $\ubar P'=0$.
	This means that the downstream pressure $\ubar P$ is a constant. Consequently, in view of \eqref{entropy}, the downstream density $\ubar\rho$ of a non-isentropic flow is a function. Same arguments hold for the upstream flow states.
\end{remark}

To prove Theorem \ref{result}, we first establish the existence of solutions to the jet problem with sufficiently large pressure at the upstream. We derive an equivalent formulation of the full Euler system in terms of the stream function, which is a second order quasilinear equation with an inhomogeneous term coming from the vorticity and the entropy. Moreover, the equation is elliptic in the subsonic region and degenerate when the flow approaches the sonic state. To deal with this possible degeneracy, we use a subsonic truncation to get a uniformly elliptic equation. A domain truncation is also used to avoid the difficulties caused by the unboundedness of the domain.
Then we reformulate the non-isentropic jet problem for compressible Euler flows with non-zero vorticity into a variational problem, which has exactly the same form as in the isentropic case (cf. \eqref{variation problem}-\eqref{K} below and (3.3)-(3.5) in \cite{LSTX2023}). Thus the existence and regularity of solutions to the variational problem can be obtained by using the same arguments as in \cite{LSTX2023}. Some fine properties of the solutions are also obtained. Then we remove the domain and subsonic truncations to get the desired solution to the original jet problem. Furthermore, the uniqueness of the solution to the jet problem and the existence of a critical upstream pressure are established.

This paper is organized as follows. In Section \ref{stream formulation and subsonic truncation}, we reformulate the jet problem for the full Euler flows in terms of the stream function formulations, and introduce the subsonic truncation. In Section \ref{variational formulation for the free boundary problem}, we use domain truncations and give the variational formulation for the truncated problems. The existence, regularity and fine properties of solutions to the truncated problems are investigated in Section \ref{sec:truncated problem}, where the continuous fit and smooth fit of the free boundary are also established. In Section \ref{sec:limit_solution}, we remove the domain and subsonic truncations, then we obtain the existence and uniqueness of the subsonic solution to the jet problem. The existence of the critical pressure is given in Section \ref{Existence of critical flux}.

In the rest of the paper, for convention the repeated indices mean the summation.

\section{Stream function formulation and subsonic truncation}\label{stream formulation and subsonic truncation}

In this section, inspired by \cite{LSTX2023,MR2607929} we introduce the stream function to show that the full Euler system is equivalent to a second order quasilinear equation, which is elliptic in the subsonic region and becomes singular at the sonic state. Then a subsonic truncation is used to modify this quasilinear equation into a uniformly elliptic equation.

\subsection{The equation for the stream function}\label{the equation for the stream function}

It follows from the continuity equation \eqref{Euler}$_1$ that there is a stream function $\psi$ satisfying
\begin{equation}\label{psi gradient}
\n\psi=(-\r u_2,\r u_1).
\end{equation}
Moreover, by \eqref{Euler}, if the flow is away from vacuum, then
\begin{equation}\label{BS}
	\u\c\n B=0 \q\text{and}\q \u\c\n S=0,
\end{equation}
where $B$ is the Bernoulli function defined by
\begin{equation}\label{Bernoulli}
	B:=\frac{|\u|^2}{2}+\frac{\g P}{(\g-1)\r},
\end{equation}
and $S$ is the entropy function defined by
\begin{equation}\label{entropy}
	S:=\frac{\g P}{(\g-1)\r^\g}.
\end{equation}
This implies the Bernoulli function $B$ and the entropy function $S$ are constants along each streamline. Now we show that actually $B$ and $S$ can be expressed as functions of the stream function $\psi$ under suitable assumptions. For this we denote the mass flux of the flow as
\begin{equation}\label{def:Q}
	Q:=\int_{0}^{\bar H}(\bar\rho\bar u)(x_2)dx_2.
\end{equation}

\begin{proposition}\label{prop:BS_def}
	Given a constant pressure $\bar P>0$, a density $\bar \rho \in C^{1,1}([0,\bar H])$ and a horizontal velocity $\bar u \in C^{1,1}([0,\bar H])$ of the flow at the upstream, which satisfy \eqref{thm_condition1}. Suppose that $(\rho, \mathbf{u},P)$ is a solution to the Euler system \eqref{Euler}, and each streamline is globally well-defined in the flow region. Then there exist $C^{1,1}$ functions $\mathcal{B}:[0,Q]\rightarrow \R$ and  $\mathcal{S}:[0,Q]\rightarrow \R$ such that
	\begin{equation}\label{BBSS}
		(B,S)(x_1,x_2)=(\B,\S)(\psi(x_1,x_2))
	\end{equation}
	and
	\begin{equation}\label{BS relation}
		\B(\psi)=\frac{|\nabla \psi|^2}{2\rho^2}+\r^{\g-1}\S(\psi).
	\end{equation}
	Moreover, denote
	\begin{equation}\label{k0}
		\kappa_0:=\|{\bar u}'\|_{C^{0,1}([0, \bar H])} \q\text{and}\q \kappa_1:=\|{\bar \rho}'\|_{C^{0,1}([0,\bar H])}.
	\end{equation}
	Then
	\begin{equation}\label{k0_BS}\begin{split}
		&\|\mathcal{B}'\|_{L^\infty([0, Q])}\leq  \frac{C}{\bar\rho_*}\left(\k_0+\frac{\bar P\k_1}{(\bar\rho_*)^2}\right),\q
		\|\mathcal{S}'\|_{L^\infty([0, Q])}\leq \frac{C\bar P\k_1}{(\bar\rho_*)^{\g+2}},
		\\
		&\|\mathcal{B}''\|_{L^\infty([0,Q])}\leq \frac{C}{(\bar\rho_*)^2}\left(\k_0+\frac{\k_0\k_1}{\bar\rho_*}+\frac{\bar P\k_1}{(\bar\rho_*)^2}\left(1+\k_0+\frac{\k_1}{\bar\rho_*}\right)\right),\\
		&\|\mathcal{S}''\|_{L^\infty([0,Q])}\leq \frac{C\bar P\k_1}{(\bar\rho_*)^{\g+3}}\left(1+\k_0+\frac{\k_1}{\bar\rho_*}\right),
	\end{split}\end{equation}
	where $C=C(\gamma,\bar u_*)>0$, $\bar\rho_*$ and $\bar u_*$ are defined in \eqref{thm_condition1}.
\end{proposition}
\begin{proof}
	Since the Bernoulli function $B$ and the entropy function $S$ are conserved along each streamline, which is globally well-defined in the flow region, then $B$ and $S$ are uniquely determined by their values at the upstream. Let $\h(z):[0,Q]\to[0,\bar H]$ be the position of the streamline at the  upstream where the stream function takes the value $\psi=z$, i.e.
\begin{equation}\label{h}
z=\int_0^{\h(z)}(\bar{\r}\bar{u})(x_2)dx_2.
\end{equation}
Note that $\h(z)$ is well-defined as $\bar\rho,\, \bar u>0$. Since the Bernoulli function and the entropy function at the upstream are
\begin{equation}\label{def:B_upstream}
	\bar B(x_2):=\lim_{x_1\rightarrow -\infty} B(x_1,x_2)=\frac{\bar u(x_2)^2}{2}+\frac{\g \bar{P}}{(\g-1)\bar{\r}(x_2)}
\end{equation}
and
\begin{equation}\label{def:B_upstream}
	\bar S(x_2):=\lim_{x_1\rightarrow -\infty} S(x_1,x_2)=\frac{\g \bar{P}}{(\g-1)\bar{\r}(x_2)^\g}
\end{equation}
for $x_2\in[0,\bar H]$, if we define two functions $\mathcal B$ and $\S$ as
\begin{equation}\label{Bpsi}
\B(z):=\bar B(\h(z))=\frac{\bar{u}(\h(z))^2}{2}+\frac{\g \bar{P}}{(\g-1)\bar{\r}(\h(z))}
\end{equation}
and
\begin{equation}\label{Spsi}
\S(z):=\bar S(\h(z))=\frac{\g\bar{P}}{(\g-1)\bar{\r}(\h(z))^\g}
\end{equation}
for $z\in[0,Q]$, then one has \eqref{BBSS}. Since by \eqref{psi gradient} the flow velocity $\u:=(u_1,u_2)$ satisfies $|\u|^2=|\n\psi|^2/\r^2$, the equality \eqref{BS relation} follows from the definitions of the Bernoulli function in \eqref{Bernoulli} and the entropy function in \eqref{entropy}, as well as \eqref{BBSS}.

To estimate the first and second derivatives of the functions $\B$ and $\S$, we note that by
differentiating \eqref{h} with respect to $z$ one gets
$$\h'(z)=\frac{1}{(\bar \r\bar u)(\h(z))}.$$
Thus differentiating \eqref{Bpsi} and \eqref{Spsi} directly yield
\begin{equation}\label{BS deri}\begin{split}
		&\B'(z)=\frac{\bar{u}'(\h(z))}{\bar{\rho}(\h(z))}-\frac{\g\bar P}{\g-1}\left(\frac{\bar\rho'}{\bar\rho^3\bar u}\right)(\h(z)),
		\q \S'(z)=-\frac{\g^2\bar P}{\g-1}\left(\frac{\bar\rho'}{\bar\rho^{\g+2}\bar u}\right)(\h(z)),\\
		&\B''(z)=\frac1{\bar\rho^2\bar u}\left[\bar u''-\frac{\bar\rho'\bar u'}{\bar\rho}-\frac{\g\bar P}{\g-1}\left(\frac{\bar\rho''}{\bar\rho^2\bar u}-\frac{3(\bar\rho')^2}{\bar\rho^3\bar u}-\frac{\bar\rho'\bar u'}{\bar\rho^2\bar u^2}\right)\right](\h(z)),\\
		&\S''(z)=-\frac{\g^2\bar P}{\g-1}\left[\frac1{\bar{\r}^{\g+3}\bar{u}^2}\left(\bar{\rho}''-\frac{\bar\rho'\bar u'}{\bar u}-(\g+2)\frac{(\bar\rho')^2}{\bar \rho}\right)\right](\h(z)).
\end{split}\end{equation}
Then \eqref{k0_BS} follows from the definitions of $\k_0$ and $\k_1$ in \eqref{k0}.
\end{proof}

Now we deduce the connection between $(\B,\S)$ and the vorticity function $\o:=\p_{x_1}u_2-\p_{x_2}u_1$. By \eqref{Bernoulli} and \eqref{entropy} one has
\begin{equation}\label{BS_relation}
	B=\frac{u_1^2+u_2^2}{2}+\left(\frac{\g P}{\g-1}\right)^{\frac{\g-1}{\g}}S^{\frac1\g}.
\end{equation}
Differentiating \eqref{BS_relation} and noting that
$$\r=\left(\frac{\g P}{(\g-1)S}\right)^{\frac1\g},$$
one gets
\begin{equation}\label{Ber pt}
\p_{x_1}B=u_1\p_{x_1}u_1+u_2\p_{x_1}u_2+\frac{\r^{\g-1}}\g\p_{x_1}S+\frac{\p_{x_1}P}{\r}
\end{equation}
and
\begin{equation}\label{Ber pt2}
 \p_{x_2}B= u_1\p_{x_2}u_1+u_2\p_{x_2}u_2+\frac{\r^{\g-1}}\g\p_{x_2}S+\frac{\p_{x_2}P}{\r}.
\end{equation}
These together with
\begin{equation}\label{B_equ}
u_1\p_{x_1}u_1+u_2\p_{x_2}u_1+\frac{\p_{x_1}P}{\r}=0
\quad\text{and}\quad  u_1\p_{x_1}u_2+u_2\p_{x_2}u_2+\frac{\p_{x_2}P}{\r}=0,
\end{equation}
which are obtained from the first three equations in \eqref{Euler}, give
\begin{equation}\label{BS_equ1}
\p_{x_1}B-\frac {\r^{\g-1}}\g \p_{x_1}S =u_2\omega\q{\rm and}\q \p_{x_2}B-\frac {\r^{\g-1}}\g \p_{x_2}S=-u_1\omega.
\end{equation}
Since it follows from \eqref{psi gradient} and \eqref{BBSS} that
\begin{equation}\label{pBS}
	(\p_{x_1}B,\p_{x_2}B)=\r \B'(\psi)(-u_2,u_1)\q{\rm and}\q 	(\p_{x_1}S,\p_{x_2}S)=\r \S'(\psi)(-u_2,u_1),
\end{equation}
substituting \eqref{pBS} into \eqref{BS_equ1} yields
\begin{equation}\label{vorticity}
	\o=-\r \B'(\psi)+\frac{\r^\g}{\g}\S'(\psi)
\end{equation}
provided $|\mathbf u|\neq 0$. Thus smooth solutions of the full Euler system  \eqref{Euler} satisfy a system consisting of the continuity equation \eqref{Euler}$_1$, \eqref{BS} and \eqref{vorticity}. The following proposition shows that the two systems are actually equivalent under appropriate conditions.

\begin{proposition}\label{prop:equivalent_system}
	Let $\tilde\O\subset \R^2$ be the domain bounded by two streamlines $\N_0$ and
	$$\tilde{\N}=\{(x_1,x_2): x_1=\tilde\Theta(x_2), \ul h<x_2<\bar H\},$$
	where $0<\ul{h}<\bar{H}<\infty$ and $\tilde \Theta:(\ul h,\bar H)\to\mathbb R$ is a $C^1$ function with
	$$\lim_{x_2\to\bar H-}\tilde\Theta(x_2)=-\infty \q{\rm and}\q \lim_{x_2\to\ul{h}+}\tilde\Theta(x_2)=\infty.$$
	Let $\rho:\overline{\tilde \Omega}\to(0,\infty)$, $\mathbf u=(u_1,u_2):\overline{\tilde \Omega}\to\mathbb R^2$ and $P:\overline{\tilde \Omega}\to(0,\infty)$
	be $C^{1,1}$ in $\tilde \Omega$ and continuous up to $\partial {\tilde \Omega}$ except finitely many points.
	Suppose $\u$ satisfies the slip boundary condition $\u\c\mn=0$ on $\p \tilde{\O}$, $(\rho,\u,P)$ satisfies the upstream asymptotic behavior  \eqref{asymptotic rup} for some positive $C^{1,1}$ functions $(\bar\rho,\bar u)$ and a positive constant $\bar P$. Moreover, suppose that
    \begin{equation}\label{u2}
	u_2<0 \q\text{in } \tilde{\O}.
    \end{equation}
	Then $(\rho,\u,P)$ solves the full Euler system \eqref{Euler} in $\tilde \Omega$ if and only if $(\rho,\u,P)$ solves the system consisting of the continuity equation \eqref{Euler}$_1$, \eqref{BS} and \eqref{vorticity} in $\tilde{\O}$.
\end{proposition}

\begin{proof}
	In view of \eqref{u2}, through each point in $\tilde{\O}$, there is one and only one streamline which satisfies
	\begin{equation*}
		\left\{ \begin{split}
			&\frac{dx_1}{ds}=u_1(x_1(s),x_2(s)),\\
			&\frac{dx_2}{ds}=u_2(x_1(s),x_2(s)),
		\end{split}
		\right.
	\end{equation*}
	and can be defined globally in the domain. Furthermore, any streamline through some point in $\tilde{\O}$ cannot touch $\p\tilde\O$. To prove it, without loss of generality we assume that there exists a streamline through $(-\bar x_1,\bar x_2)$ (with $\bar x_1>0$ sufficiently large) in $\tilde{\O}$ passing through $(\tilde x_1, 0)$. By  the continuity equation \eqref{Euler}$_1$ and the slip boundary condition, along each streamline one has
	$$0=\int_{0}^{\bar x_2}(\r u_1)(-\bar x_1,s)ds,$$
	which contradicts the asymptotic behavior \eqref{asymptotic rup}.
	As each streamline is globally well-defined in $\tilde\O$, from Proposition \ref{prop:BS_def} and previous analysis, it suffices to derive the full Euler system \eqref{Euler} from the system consisting of \eqref{Euler}$_1$, \eqref{BS} and \eqref{vorticity}.
	
	It follows from \eqref{vorticity} and \eqref{psi gradient} that
$$\r\u\c\n\left(\frac{\o}{\r}-\frac{\r^{\g-1}}\g\S'(\psi)\right)=0.$$
With the help of \eqref{psi gradient}, \eqref{pBS} and \eqref{entropy}, one can check that the above equality is equivalent to
\begin{equation}\label{equ4}
\r\u\c\n\left(\frac{\o}{\r}\right)+\n\c\left(\frac{\p_{x_2}P}{\r},-\frac{\p_{x_1}P}{\r}\right)=0.
\end{equation}
Since the continuity equation \eqref{Euler}$_1$ leads to
$$\r\u\c\n\left(\frac{\o}{\r}\right)=\u\c\n\o-\o\frac{\u\c\n\r}{\r}=\u\c\n\o+\o\n\c\u,$$
it follows from \eqref{equ4} that
$$\p_{x_1}\left(u_1\p_{x_1}u_2+u_2\p_{x_2}u_2+\frac{\p_{x_2}P}{\r}\right)-\p_{x_2}\left(u_1\p_{x_1}u_1+u_2\p_{x_2}u_1+\frac{\p_{x_1}P}{\r}\right)=0.$$
Therefore there exists a function $\Phi$ such that
$$\p_{x_1}\Phi=u_1\p_{x_1}u_1+u_2\p_{x_2}u_1+\frac{\p_{x_1}P}{\r}\q{\rm and}\q \p_{x_2}\Phi=u_1\p_{x_1}u_2+u_2\p_{x_2}u_2+\frac{\p_{x_2}P}{\r}.$$
Moreover, combining \eqref{BS}, \eqref{Ber pt} and \eqref{Ber pt2}  together yields
$$\u\c\left(u_1\p_{x_1}u_1+u_2\p_{x_2}u_1+\frac{\p_{x_1}P}{\r}, u_1\p_{x_1}u_2+u_2\p_{x_2}u_2+\frac{\p_{x_2}P}{\r}\right)=0,$$
that is
$$\u\c\n\Phi=0.$$
This means $\Phi$ is a constant along each streamline, in particular along $\N_0$ and $\tilde{\N}$. Notice that \eqref{asymptotic rup} implies $\p_{x_2}\Phi\to0$ as $x_1\to-\infty$, hence one has
\begin{equation*}
\Phi\to C\q\text{as } x_1\to-\infty.
\end{equation*}
Therefore, one concludes that $\Phi\equiv C$ in the whole domain $\tilde{\O}$. This implies that $\p_{x_1}\Phi=\p_{x_2}\Phi\equiv0$ in $\tilde{\O}$, i.e., \eqref{B_equ}
holds globally in $\tilde{\O}$.

Combining the continuity equation \eqref{Euler}$_1$,  \eqref{B_equ} and the Bernoulli law (the first equality in \eqref{BS})  together gives the full Euler system \eqref{Euler}.
\end{proof}

For later purpose, we extend the Bernoulli function $\B$ and the entropy function $\S$ from $[0,Q]$ to $\R$ as follows: firstly in view of \eqref{thm_condition1}, $\bar\rho$ and $\bar u$ can be extended to $C^{1,1}$ functions in $\R$, which are still denoted by $\bar\rho$ and $\bar u$ respectively, such that
$$\bar\rho,\,\bar u>0 \ \text{ on } \R, \quad {\bar\rho'=\bar u'= 0 \  \text{ on } (-\infty,0],}
\quad \bar\rho'=0, \, \bar u'\geq 0
\ \text{ on } [\bar H,\infty).$$
Furthermore, the extension can be made such that
$$\|\bar \rho\|_{C^{1,1}(\R)}=\|\bar \rho\|_{C^{1,1}([0,\bar H])}\q \text{and}\q
\|\bar u\|_{C^{1,1}(\R)}\leq C\|\bar u\|_{C^{1,1}([0,\bar H])}$$
and
\begin{align}\label{rhou_R_upper}
	\bar \rho_\ast=\inf_{\R}\bar\rho \leq \sup_{\R}\bar \rho=\bar \rho^\ast\q\text{and}\q
	\bar u_\ast=\inf_{\R}\bar u \leq \sup_{\R} \bar u:=\tilde u^\ast<C\bar u^*,
\end{align}
where the constants $\bar\rho_*$ and $\bar u_*$ are defined in \eqref{thm_condition1}, $\bar\rho^*$ and $\bar u^*$ are defined by
\begin{align}\label{rhou_bar_upper}
	\bar\rho^*:=\sup_{x_2\in[0,\bar H]} \bar \rho(x_2)\q\text{and}\q
	\bar u^*:=\sup_{x_2\in[0,\bar H]}\bar u(x_2),
\end{align}
and $C=C(\bar u)>0$.
Consequently, using \eqref{Bpsi} and \eqref{Spsi} one naturally gets extensions of $\mathcal B$ and $\S$ to $C^{1,1}$ functions in $\R$, which are still denoted by $\mathcal{B}$ and $\S$ respectively, such that
\begin{equation}\label{eq:sign_B}
	\mathcal B'(z)=\mathcal S'(z)=0 \ \text{ on } (-\infty, 0], \q \B'(z)\geq0, \, \S'(z)=0 \ \text{ on } [Q,\infty),
\end{equation}
and
\begin{equation*}\label{eq:B_derivative}
	\|\mathcal{B}'\|_{C^{0,1}(\R)}\leq C \|\mathcal{B}'\|_{C^{0,1}([0, Q])}, \q \|\mathcal{S}'\|_{C^{0,1}(\R)}= \|\mathcal{S}'\|_{C^{0,1}([0, Q])}
\end{equation*}
for some $C=C(\bar u)>0$. Moreover, $\B$ and $\S$ are bounded in $\R$, i.e.,
\begin{align}\label{BS_bound}
	0<B_*\leq\B(z)\leq B^*<\infty \q\text{and}\q
	0<S_*\leq\S(z)\leq S^*<\infty, \q z\in \R,
\end{align}
where
\begin{align}\label{B**}
	B_*:=\frac{(\bar{u}_*)^2}{2}+\frac{\g \bar{P}}{(\g-1)\bar{\r}^*},\q
	B^*:=\frac{(\tilde{u}^*)^2}{2}+\frac{\g \bar{P}}{(\g-1)\bar{\r}_*},
\end{align}
and
\begin{align}\label{S**}
	S_*:=\frac{\g\bar{P}}{(\g-1)(\bar{\r}^*)^\g},
	\q
	S^*:=\frac{\g\bar{P}}{(\g-1)(\bar{\r}_*)^\g},
\end{align}
and where $\bar\rho^*$, $\bar\rho_*$, $\bar u_*$, $\tilde u^*$ are the same constants as in \eqref{rhou_R_upper}.
	
Before formulating the Euler system into a quasilinear equation for the stream function $\psi$, let us digress for the study on the flow state.
It follows from \eqref{Bernoulli}-\eqref{entropy} and Proposition \ref{prop:BS_def} that the flow speed $q$ satisfy
\begin{equation}\label{eq:q_rho}
	q=\mathfrak q(\rho,z)=\sqrt{2(\B(z)-\r^{\g-1}\S(z))}.
\end{equation}
Besides, using \eqref{entropy} and the definition of the sound speed in \eqref{C and M} one has
$$c=\sqrt{(\g-1)\rho^{\g-1}\mathcal S(z)}.$$
Let
\begin{equation}\label{rho cm}
	\varrho_c(z):=\left(\frac{2\B(z)}{(\g+1)\S(z)}\right)^{\frac1{\g-1}} \q\text{and}\q \varrho_m(z):=\left(\frac{\B(z)}{\S(z)}\right)^{\frac1{\g-1}}
\end{equation}
be the critical density and maximum density respectively. The flow speed $q=\mathfrak q(\rho,z)$ is well-defined if and only if $\rho\leq \varrho_m(z)$. The flow is subsonic, i.e., $q\in[0,c)$, if and only if
$q\in[0,\mathfrak{q}(\varrho_c(z),z))$, or equivalently   $\r\in(\varrho_c(z),\varrho_m(z)]$. Denote the square of the momentum and the square of the critical momentum by
\begin{equation}\label{t}
	\t(\rho, z):=\rho^2\mathfrak{q}^2(\rho, z)=2\rho^2(\B(z)-\rho^{\g-1}\S(z))
\end{equation}
and
\begin{equation}\label{tc}
	\t_c(z):=\t(\varrho_c(z),z)=(\g-1)\left(\frac{2\B(z)}{\g+1}\right)^{\frac{\g+1}{\g-1}}\S(z)^{-\frac2{\g-1}},
\end{equation}
respectively.

By the definition of the sonic speed in \eqref{C and M}, the flow is subsonic at the upstream if
$(\tilde u^*)^2<\g\bar P/\bar\rho^*$, where $\tilde u^*$ is defined in \eqref{rhou_R_upper}. This inequality holds provided
\begin{equation}\label{Pbar_*}
	\bar P>\tilde P_*, \quad\text{where}\q \tilde P_*:=\frac1{\g}\bar\rho^*(\tilde u^*)^2\geq\bar P_*:=\frac1{\g}\sup_{x_2\in[0,\bar H]}(\bar\rho\bar u^2)(x_2).
\end{equation}
Then by \eqref{B**} and \eqref{S**} one has
\begin{equation}\label{B/S}
	B^*<\frac{(\g+1)\g\bar P}{2(\g-1)\bar\rho_*} \q\text{and}\q
	(\bar \rho_*)^{\g-1}<\frac{\B(z)}{\S(z)}<\frac{\g+1}{2}(\bar\rho^*)^{\g-1}, \q z\in\R.
\end{equation}
Thus in view of \eqref{rho cm} and \eqref{tc}, for $z\in \R$,
\begin{align}\label{rho_bound}
	0<\left(\frac{2}{\g+1}\right)^{\frac1{\g-1}}\bar\rho_*<\varrho_c(z)<\varrho_m(z)<\left(\frac{\g+1}2\right)^{\frac1{\g-1}}\bar\rho^*,
\end{align}
and
\begin{align}\label{tc_bound}
0<\left(\frac2{\g+1}\right)^{\frac{\g+1}{\g-1}}\g\bar P\frac{(\bar\rho_*)^2}{\bar\rho^*}\leq\mathfrak t_c(z)\leq \g\bar P\frac{(\bar\rho^*)^2}{\bar\rho_*}.
\end{align}
Furthermore, since
\begin{equation}\label{tc_deri_expression}
	\t_c'(z)=\left(\frac2{\gamma+1}\right)^{\frac{\gamma+1}{\gamma-1}}\left((\gamma+1)\left(\frac{\B(z)}{\S(z)}\right)^{\frac2{\gamma-1}}\B'(z)-2\left(\frac{\B(z)}{\S(z)}\right)^{\frac{\gamma+1}{\gamma-1}}\S'(z)\right),
\end{equation}
by the $L^\infty$-estimates for the derivatives of $\B$ and $\S$ in \eqref{k0_BS}, and the upper bound of $\B/\S$ in \eqref{B/S}, one has
\begin{equation}\label{tc_deri1}
		|\t'_c(z)|\leq C\frac{(\bar\rho^*)^2}{\bar\rho_*}\left(\k_0+\frac{\bar P\k_1}{(\bar\rho_*)^2}\left(1+\left(\frac{\bar\rho^*}{\bar\rho_*}\right)^{\g-1}\right)\right)
\end{equation}
for some constant $C=C(\g,\bar u_*)>0$.
Similarly, by the explicit expression of $\t_c''(z)$ one has
\begin{equation}\label{tc_deri2}\begin{split}
		|\t''(z)|\leq& C\frac{(\bar\rho^*)^2}{(\bar\rho_*)^2}\left(\k_0+\frac{\k_0\k_1}{\bar\rho_*}+\frac{\bar P\k_1}{(\bar\rho_*)^2}\left(1+\k_0+\frac{\k_1}{\bar\rho_*}\right)\left(1+\left(\frac{\bar\rho^*}{\bar\rho_*}\right)^{\g-1}\right)\right)\\
		&+C\frac{(\bar\rho^*)^{\g+2}}{\bar P(\bar\rho_*)^{\g+1}}\left(\k_0+\frac{\bar P\k_1}{(\bar\rho_*)^2}\left(1+\left(\frac{\bar\rho^*}{\bar\rho_*}\right)^{\g-1}\right)\right)\left(\k_0+\frac{\bar P\k_1}{(\bar\rho_*)^2}\right),
	\end{split}
\end{equation}
for some constant $C=C(\g,\bar u_*)>0$.

\begin{lemma}\label{g}
	Suppose the density function $\r$ and the stream function $\psi$ satisfy \eqref{BS relation}. Then the following statements hold:
	\begin{itemize}
		\item [(i)] The density function $\r$ can be expressed as a function of $|\n\psi|^2$ and $\psi$ in the subsonic region, i.e.,
		\begin{equation}\label{rho g}
			\r=\frac 1{g(|\n\psi|^2,\psi)}, \q \text{if } \r\in (\varrho_c(\psi),\varrho_m(\psi)],
		\end{equation}
		where $\varrho_c$ and $\varrho_m$ are functions defined in \eqref{rho cm}, and
		$$g:\{(t,z):0\leq t<\mathfrak t_c(z),\, z\in \R\}\to \R$$
		is a function smooth in $t$ and $C^{1,1}$ in $z$ with $\mathfrak t_c$ defined in \eqref{tc}. Furthermore,
		\begin{equation}\label{g bound}
			\left(\sup_{z\in\R}\varrho_m(z)\right)^{-1}=:g_*\leq g(t,z)\leq g^*:=\left(\inf_{z\in\R}\varrho_c(z)\right)^{-1}.
		\end{equation}
		\item [(ii)] The function $g$ satisfies the identity
		\begin{equation}\label{dzg dtg}
			\p_z g(t,z)=-2\p_t g(t,z)\frac{\B'(z)-g(t,z)^{1-\g}\S'(z)}{g(t,z)^2}, \q t\in[0,\mathfrak t_c(z)),\, z\in \R.
		\end{equation}
	\end{itemize}
\end{lemma}

\begin{proof}
\emph{(i).} From the expression \eqref{t}, straightforward computations give
\begin{equation}\label{dF}
\p_\r\t(\r, z)=4\r\left(\B(z)-\frac{\g+1}2\r^{\g-1}\S(z)\right).
\end{equation}
Now, with $\varrho_c$ and $\varrho_m$ defined in (\ref{rho cm}) one has
\begin{enumerate}
  \item[\rm (a)] $\r\mapsto\t(\r,z)$ achieves its maximum $\t_c(z)$ at $\r=\varrho_c(z)$;
  \item[\rm (b)] $\p_\r\t(\r,z)<0$ when $\varrho_c(z)<\r\leq\varrho_m(z)$.
\end{enumerate}
Thus by the inverse function theorem, if $\rho\in(\varrho_c(z),\varrho_m(z)]$, one can express $\rho$ as a function of $t:=\t(\rho,z)\in[0,\t_c(z))$ and $z$, that is, $\rho=\rho(t,z)$. Let
\begin{equation}\label{g_def}
g(t,z):=\frac1{\r(t,z)}, \q t\in[0,\mathfrak t_c(z)),\, z\in \R.
\end{equation}
The function $g$ is smooth in $t$ by the inverse function theorem and $C^{1,1}$ in $z$ by the $C^{1,1}$ regularity of $\B$ and $\S$.
This together with \eqref{BS relation} gives \eqref{rho g}. Consequently \eqref{g bound} holds.

\emph{(ii).}
Differentiating \eqref{t} with respect to $t:=\t(\rho,z)$ and $z$, respectively, one gets
\begin{equation}\label{bdt}
	\frac12=\frac{\p_t\r}{\r}(t-(\g-1)\r^{\g+1}\S(z))
\end{equation}
and
\begin{equation}\label{bdz}
	\frac{\p_z\r}{\r}(t-(\g-1)\r^{\g+1}\S(z))=-\r^2(\B'(z)-\r^{\g-1}\S'(z)).
\end{equation}
Combining \eqref{bdt} and \eqref{bdz} together gives
$$\p_z\r=-2\r^2\p_t\r(\B'(z)-\r^{\g-1}\S'(z)).$$
Thus by \eqref{rho g}, the equality \eqref{dzg dtg} holds. This completes the proof of the lemma.
\end{proof}

The following lemma gives the equation of the stream function in the subsonic region.

\begin{lemma}\label{stream formulation lem}
Let $(\r,\u,P)$ be a solution to the system consisting of  \eqref{Euler}$_1$, \eqref{BS} and \eqref{vorticity}. Assume $(\r,\u,P)$ satisfies the assumptions in Proposition \ref{prop:equivalent_system}. Then in the subsonic region $|\nabla\psi|^2<\t_c(\psi)$ where $\t_c$ is defined in \eqref{tc}, the stream function $\psi$ solves
\begin{equation}\label{elliptic equ}
\n\c(g(|\n \psi|^2,\psi)\n \psi)=\frac{\B'(\psi)}{g(|\n \psi|^2,\psi)}-\frac{\S'(\psi)}{\g g(|\n \psi|^2,\psi)^\g},
\end{equation}
where $g$ is defined in \eqref{g_def}, $\B$ is the Bernoulli function  defined in \eqref{Bpsi}, and $\S$ is the entropy function defined in \eqref{Spsi}. Moreover, the equation \eqref{elliptic equ} is elliptic if and only if $|\n\psi|^2<\t_c(\psi)$.
\end{lemma}

\begin{proof}
Expressing the vorticity $\o$ in terms of the stream function $\psi$, and using \eqref{vorticity} one has
\begin{equation*}
	-\n\c\Big(\frac{\n\psi}{\r}\Big)=\o=-\r\B'(\psi)+\frac{\r^\g}{\g}\S'(\psi).
\end{equation*}
In view of \eqref{rho g} the above equation can be rewritten as  \eqref{elliptic equ}.

The equation \eqref{elliptic equ} can be written in the nondivergence form as follows
$$a^{ij}(\n\psi,\psi)\p_{ij}\psi+\p_zg(|\n\psi|^2,\psi)|\n\psi|^2=\frac{\B'(\psi)}{g(|\n \psi|^2,\psi)}-\frac{\S'(\psi)}{\g g(|\n \psi|^2,\psi)^\g},$$
where the matrix
$$(a^{ij})=g(|\n\psi|^2,\psi)I_2+2\p_t g(|\n\psi|^2,\psi)\n\psi\otimes\n\psi$$
is symmetric with the eigenvalues
$$\b_0=g(|\n \psi|^2,\psi) \q{\rm and}\q \b_1=g(|\n\psi|^2,\psi)+2\p_tg(|\n\psi|^2,\psi)|\n\psi|^2.$$
Differentiating the identity $t=\t(\frac{1}{g(t,z)},z)$ gives
\begin{equation}\label{dtg}
\p_tg(t,z)=-\frac{g(t,z)^2}{\p_\r\t(\frac1{g(t,z)},z)} \q\text{for } t\in[0,\t_c(z)),\, z\in\R.
\end{equation}
This implies
$$\p_tg(t,z)\geq0 \q {\rm and}\q \lim_{t\to\t_c(z)-}\p_tg(t,z)=+\infty.$$
Thus $\b_0$ has uniform upper and lower bounds
(cf. \eqref{g bound} and \eqref{rho_bound}),
and $\b_1$ has a uniform lower bound but blows up when $|\n\psi|^2$ approaches $\t_c(\psi)$. Therefore, the equation (\ref{elliptic equ}) is elliptic as long as $|\n\psi|^2<\t_c(\psi)$, and is singular when $|\n\psi|^2=\t_c(\psi)$. This completes the proof of the lemma.
\end{proof}

\subsection{Reformulation for the jet flows in terms of the stream function}

Let
\begin{equation}\label{Ld}
\Ld:=\r_0\sqrt{2\B(Q)-2\r_0^{\g-1}\S(Q)} \q {\rm with}\q \r_0:=\left(\frac{\g\ul P}{(\g-1)\S(Q)}\right)^{\frac1\g}
\end{equation}
be the constant momentum on the free boundary, where
$\B$ is defined in \eqref{Bpsi}, $\S$ is defined in \eqref{Spsi}, and 
$\ul P>0$ is the pressure on the free boundary as in Problem \ref{problem1}.
Then Problem {\ref{problem1}} can be solved as long as the following problem in terms of the stream function is solved.

\begin{problem}\label{jet problem}
Given a density $\bar\rho\in C^{1,1}([0,\bar H])$ and a horizontal velocity $\bar u\in C^{1,1}([0,\bar H])$ satisfying \eqref{thm_condition1}. Given a constant pressure $\bar P>0$ of the flow at the upstream. Find a triple $(\psi,\G_\psi,\Ld)$ satisfying $\psi\in C^{2,\alpha}(\{\psi<Q\})\cap C^{1}(\overline{\{\psi<Q\}})$ ($\alpha\in(0,1)$), $\p_{x_1}\psi>0$ in $\{0<\psi<Q\}$ and
\begin{equation}\label{jet equ}
\begin{cases}\begin{split}
&\n\c(g(|\n \psi|^2,\psi)\n \psi)=\frac{\B'(\psi)}{g(|\n \psi|^2,\psi)}-\frac{\S'(\psi)}{\g g(|\n \psi|^2,\psi)^\g} &&\text{ in } \{\psi<Q\},\\
&\psi=0 &&\text{ on } \N_0,\\
&\psi=Q &&\text{ on } \N_1\cup \G_\psi,\\
&|\n\psi|=\Ld &&\text{ on } \G_\psi,
\end{split}\end{cases}\end{equation}
where the free boundary $\G_\psi:=\p\{\psi<Q\}\backslash \N_1$. Furthermore, the free boundary $\G_\psi$ and the flow
\begin{equation}\label{rup_psi}
	(\r, \u, P)=\left(\frac1{g(|\n\psi|^2,\psi)}, \, g(|\n\psi|^2,\psi)\nabla^{\perp}\psi, \, \frac{(\gamma-1)\S(\psi)}{\gamma g(|\n\psi|^2,\psi)^\gamma}\right)
\end{equation}
are expected to have the following properties:
\begin{enumerate}
   \item The free boundary $\Gamma_\psi$ is given by a graph $x_1=\Upsilon(x_2)$, $x_2\in (\ubar{H}, 1]$ for some $C^{2,\alpha}$  function $\Upsilon$ and some $\ubar{H}\in (0,1)$.
	
	\item The free boundary $\Gamma_\psi$ fits the nozzle at $A=(0,1)$ continuous differentiably, i.e. $\Upsilon(1)=\Theta(1)$ and $\Upsilon'(1)=\Theta'(1)$.

	\item For $x_1$ sufficiently large, the free boundary is also an $x_1$-graph, i.e., it can be written as $x_2=f(x_1)$ for some $C^{2,\alpha}$ function $f$. Furthermore, at downstream $x_1\to\infty$, one has
	$$\lim_{x_1\to \infty}f(x_1)=\ul{H} \q\text{and}\q \lim_{x_1\to\infty}f'(x_1)=0.$$
	
	\item The flow is subsonic in the flow region, i.e. $|\nabla\psi|^2<\t_c(\psi)$ in $\{\psi<Q\}$, where $\t_c$ is defined in \eqref{tc}.
	
	\item At the upstream the flow satisfies the asymptotic behavior \eqref{asymptotic rup}. At the downstream the flow satisfies 
	$$\lim_{x_1\to\infty}(\r,\u,P)=(\ul{\r},\ul{u},0,\ul P)$$
	for some positive functions $\ul{\r}=\ul{\r}(x_2)$ and $\ul{u}=\ul{u}(x_2)$, and some constant $\ul P>0$.
	
\end{enumerate}
\end{problem}

As in \cite[Proposition 2.5]{LSTX2023}, we have the following observation, which asserts that $(\rho,\mathbf u,P)$ defined in \eqref{rup_psi} is a solution of the Euler system \eqref{Euler} even if $\rho,\mathbf u,P$ are only $C^{1,\alpha}$.
 
\begin{proposition}\label{prop:equiv_sol}
Given a density $\bar\rho\in C^{1,1}([0,\bar H])$ and a horizontal velocity $\bar u\in C^{1,1}([0,\bar H])$ satisfying \eqref{thm_condition1}. Given a constant pressure $\bar P>0$ of the flow at the upstream. Assume that a function $\psi\in C^{2,\alpha}(\{\psi<Q\})\cap C^{1}(\overline{\{\psi<Q\}})$ is a subsonic solution of \eqref{jet equ}, where $\alpha\in(0,1)$. Then the flow $(\rho,\mathbf u,P)$ defined in \eqref{rup_psi} solves the Euler system \eqref{Euler} in the flow region $\{\psi<Q\}$.
\end{proposition}
\begin{proof}
	Due to the regularity of the function $\psi$,  the functions $\rho$, $\mathbf u:=(u_1,u_2)$ and $P$ defined in \eqref{rup_psi} belong to $C^{1,\alpha}(\{\psi<Q\})\cap C(\overline{\{\psi<Q\}})$. Since $\rho$ and $\mathbf u$  satisfy \eqref{psi gradient}, it follows from the definition of the function $g$ in Lemma \ref{g} (i) that  the density $\rho=1/g$ satisfies
	\begin{align*}
		\frac{u_1^2+u_2^2}{2}+\r^{\g-1}\S(\psi)=	\frac{|\nabla\psi|^2}{2\rho^2}+\r^{\g-1}\S(\psi)=\mathcal{B}(\psi)
	\end{align*}
	in the flow region.
	Differentiating the above equality with respect to $x_1$ and $x_2$, respectively, one gets
	\begin{equation}\label{eq1}
		u_1\partial_{x_1}u_1+u_{2}\partial_{x_1}u_2 +\frac{\r^{\g-1}}\g\S'(\psi)\p_{x_1}\psi+\frac{\p_{x_1}P}{\r}=\B'(\psi)\partial_{x_1}\psi,
	\end{equation}
	and
	\begin{align}\label{eq2}
		u_1\partial_{x_2}u_1+u_2\partial_{x_2}u_{2}+\frac{\r^{\g-1}}\g\S'(\psi)\p_{x_2}\psi+\frac{\p_{x_2}P}{\r}=\B'(\psi)\partial_{x_2}\psi.
	\end{align}
	Note that by the quasilinear equation in \eqref{jet equ} one has
	$$\B'(\psi)-\frac{\r^{\g-1}}{\g}\S'(\psi)=-\frac{\p_{x_1}u_2-\p_{x_2}u_1}{\rho}.$$
	Thus the equalities \eqref{eq1} and \eqref{eq2} together with \eqref{psi gradient} yield \eqref{B_equ}. Note that $\rho$ and $\mathbf u$ satisfy the continuity equation \eqref{Euler}$_1$. Then the Euler system \eqref{Euler} follows from \eqref{Euler}$_1$, \eqref{B_equ} and the equality $\rho\mathbf u\cdot \nabla \B(\psi)=0$ where $\B(\psi)$ satisfies \eqref{BS relation}. This finishes the proof of the proposition.
\end{proof}

\subsection{Subsonic truncation}\label{subsec_subsonic truncation}

One of the major difficulties to solve Problem \ref{jet problem} is that the equation \eqref{elliptic equ} becomes degenerate as the flow approaches the sonic state, so we use a subsonic truncation to avoid this degeneracy. Let $\varpi:\R\to[0,1]$ be a smooth nonincreasing function such that
$$\varpi(s)=\begin{cases}
1& \text{if } s\leq-1,\\
0& \text{if } s\leq-\frac12,
\end{cases}
\q
{\rm and}\q
|\varpi'|+|\varpi''|\leq8.$$
For fixed $\v\in(0,1/4)$, 
let $\varpi_\v(s):=\varpi((s-1)/\v)$. Define
\begin{equation}\label{g modified}
g_\v(t,z)=g(t,z)\varpi_\v(t/\t_c(z))+(1-\varpi_\v(t/\t_c(z)))g^*,
\end{equation}
where $\t_c(z)$ is given by \eqref{tc} and $g^*$ is the upper bound of $g$ in \eqref{g bound}. The properties of $g_\v$ are summarized in the following lemma.

\begin{lemma}\label{g properties}
Let $g$ be the function defined in Lemma \ref{g} and let $g_\v$ be the subsonic truncation of $g$ defined in (\ref{g modified}). Then the function $g_\v(t,z)$ is smooth with respect to $t$ and $C^{1,1}$ with respect to $z$. Furthermore, suppose the upstream  pressure $\bar P$ of the flow satisfies \eqref{Pbar_*}. There exist positive constants $c_*=c_*(\g)$ and $c^*=c^*(\g,\bar u_*)$ such that for all $t\geq 0$, one has
  \begin{align}
  &c_*(\bar\rho^*)^{-1}\leq g_\v(t,z)\leq (c_*\bar\rho_*)^{-1}, \label{gm}\\
  &c_*(\bar\rho^*)^{-1}\leq g_\v(t,z)+2\p_tg_\v(t,z)t\leq (c_*\v\bar\rho_*)^{-1}, \label{g dtgmt}\\
  &|\p_zg_\v(t,z)|\leq  c^*\v^{-1}\frac{(\bar\rho^*)^3}{\bar P(\bar\rho_*)^4}\left(\k_0+\frac{\bar P\k_1}{(\bar\rho_*)^2}\left(1+\left(\frac{\bar\rho^*}{\bar\rho_*}\right)^{\g-1}\right)\right),\label{dzgm1}\\
  &|t\p_zg_\v(t,z)|\leq c^*\v^{-1}\left(\frac{\bar\rho^*}{\bar\rho_*}\right)^5\left(\k_0+\frac{\bar P\k_1}{(\bar\rho_*)^2}\left(1+\left(\frac{\bar\rho^*}{\bar\rho_*}\right)^{\g-1}\right)\right),\label{dzgm2}
  \end{align}
 where $\bar\rho_*$ is defined in \eqref{thm_condition1}, and $\bar\rho^*$ is defined in \eqref{rhou_bar_upper}.
\end{lemma}

\begin{proof}
	
	It follows from Lemma \ref{g} and the definition of $\varpi_\v$ that $g_\v$ is smooth with  respect to $t$ and $C^{1,1}$ with respect to $z$.
	
    \emph{(i).}  Clearly \eqref{gm} follows directly from \eqref{g bound} and \eqref{g modified}.

	To show \eqref{g dtgmt},  we first claim that if $0\leq t\leq\left(1-\frac{\epsilon}2\right) \t_c(z)$ then \begin{equation}\label{eq:g_dt_bound}
		0<\p_t g(t,z)\leq \frac{g(t,z)}{\epsilon \t_c(z)},
	\end{equation}
	{where $\t_c$ is defined in \eqref{tc}.} In view of the expression of $\p_t g(t,z)$ in \eqref{dtg} it suffices to estimate $\p_{\rho}\t(\rho,z)$ at $\rho=\rho(t,z)$ for $0\leq t\leq\left(1-\frac{\epsilon}2\right) \t_c(z)$. In fact, it follows from \eqref{t} and \eqref{dF} that for $0\leq t\leq\left(1-\frac{\epsilon}2\right) \t_c(z)$  and $\rho\geq \varrho_c(z)$ it holds that
	\begin{align*}
			\p_\r\t\left(\rho,z\right)
		&{=}\frac{2}{\rho}\left(\t\left(\rho,z\right)-(\g-1)\rho^{\g+1}\S(z)\right)\leq\frac{2}{\rho}\left(\t\left(\rho,z\right)-(\g-1)\varrho_c(z)^{\g+1}\S(z)\right)\\
		&= \frac{2}{\rho}\left(\t\left(\rho,z\right)-\t_c(z)\right)\leq-\frac{\v\t_c(z)}{\rho}.
	\end{align*}
	Thus from \eqref{dtg} we conclude that
	\begin{equation*}
		0<\p_tg(t,z)=-\frac{g^2(t,z)}{\p_\rho\t(\rho, z)}\Big|_{\rho=\frac{1}{g(t,z)}}\leq \frac{g(t,z)}{\epsilon\t_c(z)},
	\end{equation*}
	which is the claimed inequality \eqref{eq:g_dt_bound}.
	With \eqref{eq:g_dt_bound} at hand, we thus have
	\begin{equation}\label{eq:g_dt}
		\begin{split}
			0\leq\p_tg_\epsilon(t,z) &= \partial_t g (t, z) \varpi_\epsilon\left(\frac{t}{\t_c(z)}\right)+(g(t, z)-g^*) \varpi'_\epsilon\left(\frac{t}{\t_c(z)}\right)\frac{1}{\t_c(z)}\\
			&\leq \frac{g(t,z)}{\epsilon \t_c(z)}+\frac{8g^\ast}{\epsilon \t_c(z)} \leq \frac{9 g^\ast}{\epsilon \t_c(z)}.
		\end{split}
	\end{equation}
	Note that
	\begin{equation}\label{gmdt_support}
		\p_tg_\epsilon(t,z)=0 \q \text{ in }\{(t,z):t>(1-\epsilon/2)\mathfrak t_c(z),\, z\in \R\}.
	\end{equation}
	Thus
	\begin{align*}
		0\leq t\p_tg_\epsilon(t,z)\leq\frac{9 g^\ast \t_c(z)}{\epsilon \t_c(z)}= \frac{9 g^\ast}{\epsilon}.
	\end{align*}
	Then the estimate \eqref{g dtgmt} follows directly from  \eqref{gm}.
	
	\emph{(ii).}
	 Direct computations yield
	\begin{equation}\label{gm_dz}
		\begin{aligned}
			&\p_z g_\epsilon (t,z)= \p_z g(t,z)\varpi_\epsilon\left(\frac{t}{\t_c(z)}\right) + (g^\ast-g(t,z))\varpi_\epsilon'\left(\frac{t}{\t_c(z)}\right) \frac{t\t_c'(z)}{\t_c(z)^2}\\
			\stackrel{\eqref{dzg dtg}}{=} & -2\p_t g(t,z)\frac{\B'(z)-g(t,z)^{1-\g}\S'(z)}{g(t,z)^2} \varpi_\epsilon\left(\frac{t}{\t_c(z)}\right)+ (g^\ast-g(t,z))\varpi_\epsilon'\left(\frac{t}{\t_c(z)}\right) \frac{t\t_c'(z)}{\t_c(z)^2}.
		\end{aligned}
	\end{equation}
	Note that both sums in the expression of $\p_tg_\epsilon$, cf. \eqref{eq:g_dt}, are nonnegative. Thus we get from  the expression of $\p_zg_\epsilon$ in \eqref{gm_dz} that
	\begin{equation}\label{eq:dzgm_dtgm}\begin{split}
			|\p_zg_\epsilon(t,z)|\leq& \max\left\{\frac{2|\B'(z)-g(t,z)^{1-\g}\S'(z)|}{g(t,z)^2},\frac{t\mathfrak t_c'(z)}{\mathfrak t_c(z)}\right\} \p_tg_\epsilon(t,z).
		\end{split}
	\end{equation}
	By the $L^\infty$-bounds of $\B'$ and $\S'$ in \eqref{k0_BS}, the bounds of $g$ in \eqref{g bound} and \eqref{rho_bound}, one has
	\begin{align}\label{lable_2}
	\frac{2|\B'(z)-g(t,z)^{1-\g}\S'(z)|}{g(t,z)^2}\leq \frac{C(\bar\rho^*)^2}{\bar\rho_*}\left(\k_0+\frac{\bar P\k_1}{(\bar\rho_*)^2}\left(1+\left(\frac{\bar\rho^*}{\bar\rho_*}\right)^{\g-1}\right)\right),
	\end{align}
	where $C=C(\g,\bar u_*)$.
	This together with the estimate for $|\t_c'|$ in \eqref{tc_deri1}, the estimate for $\p_tg_\v$ in \eqref{eq:g_dt}, the lower bound of $\t_c$ in \eqref{tc_bound}, and the support condition  of $\p_tg_\epsilon$ in \eqref{gmdt_support} yields  \eqref{dzgm1}.
	To show \eqref{dzgm2}, first we note that from \eqref{eq:dzgm_dtgm} and \eqref{gmdt_support} one has
	\begin{equation}\label{label_1}
		\p_zg_\epsilon(t,z)=0 \q \text{ in }\{(t,z):t>(1-\epsilon/2)\mathfrak t_c(z),\, z\in \R\}.
	\end{equation}
	Then \eqref{dzgm2} follows from \eqref{dzgm1} and the upper bound of $\t_c$ in \eqref{tc_bound}. 
	This finishes the proof of the lemma.
\end{proof}

\section{Variational formulation for the truncated  problem}\label{variational formulation for the free boundary problem}

In this section we show that the equation \eqref{elliptic equ}, which has a complicated memory term, can be written as an Euler-Lagrange equation of an energy functional. Thus  Problem \ref{jet problem} can be transformed into a domain variation problem as in the isentropic case (cf. \cite{LSTX2023}). This is the key to take advantage of the variational method developed by Alt, Caffarelli and Friedman to solve the jet problem.

In this section, we always assume that the upstream pressure $\bar P$ of the flow satisfies \eqref{Pbar_*}.

Let $\O$ be the domain bounded by $\N_0$ and $\N_1\cup([0,\infty)\times\{1\})$. Noticing that $\O$ is unbounded, we make an approximation by considering the problem in a series of truncated domains $\O_{\mu,R}:=\O\cap\{-\mu<x_1<R\}$, where $\mu$ and $R$ are two large positive numbers. To define the energy functional, we set
\begin{equation}\label{G}
G_\v(t,z):=\frac12\int_0^tg_\v(\tau,z)d\tau+\frac{\g-1}{\g}(g_\v(0,z)^{-\g}\S(z)-g_\v(0,Q)^{-\g}\S(Q)),
\end{equation}
which is smooth in $t$ and $C^{1,1}$ in $z$  by the regularity of $g_\v$ (cf. Lemma \ref{g properties}), and set
$$\Phi_\v(t,z):=-G_\v(t,z)+2\p_tG_\v(t,z)t.$$
Given a function $\psi^{\sharp}_{\mu,R}\in C(\p \O_{\mu,R})\cap H^{1}(\Omega_{\mu,R})$ with $0\leq\psi^{\sharp}_{\mu,R}\leq Q$, we consider the minimization problem:
\begin{equation}\label{variation problem}
	\text{find }\psi\in \K_{\psi^{\sharp}_{\mu,R}}  \ \text{ s.t. } \ J_{\mu,R,\Ld}^\v(\psi)= \inf_{\varphi\in \K_{\psi^{\sharp}_{\mu,R}}} J_{\mu,R,\Ld}^\v(\varphi),
\end{equation}
where the energy functional
\begin{equation}\label{J}
	J^\v_{\mu,R,\Ld}(\varphi):=\int_{\Omega_{\mu,R}}G_\v(|\n \varphi|^2,\varphi)+\ld_\v^2\chi_{\{\varphi<Q\}}dx, \q\q \ld_\v: =\sqrt{\Phi_\v(\Ld^2,Q)},
\end{equation}
with $\Ld$ defined in \eqref{Ld}, and the admissible set
\begin{equation}\label{K}
	\K_{\psi^{\sharp}_{\mu,R}}:=\{\varphi\in H^1(\O_{\mu,R}):\varphi=\psi^{\sharp}_{\mu,R}\text{ on } \p\O_{\mu,R}\}.
\end{equation}
Note that $\Phi_\v(\Ld^2,Q)>0$. Indeed, with the aid of the elliptic condition \eqref{g dtgmt}, a straightforward computation shows that $t\mapsto\Phi_\v(t,z)$ is monotone increasing:
\begin{align}\label{Phi_dt}
	0<\frac12c_*(\bar\rho^*)^{-1}\leq\p_t\Phi_\v(t,z)=\frac12g_\v(t,z)+\p_tg_\v(t,z)t\leq\frac12 (c_*\v\bar\rho_*)^{-1}.
\end{align}
Since $\Phi_\v(0,Q)=-G_\v(0,Q)=0,$ one has $\Phi_\v(t,Q)>0$ for all $t>0$.

The existence and H\"{o}lder regularity of minimizers for \eqref{variation problem} will be shown in Section \ref{sec:truncated problem}. Let us assume the existence and continuity of a minimizer $\psi$ firstly. We derive the Euler-Lagrange equation for the variational problem \eqref{variation problem} in the open set $\{\psi<Q\}$.

\begin{lemma}\label{EL}
Let $\psi$ be a minimizer of (\ref{variation problem}). Assume that $\O_{\mu,R}\cap\{\psi<Q\}$ is open. Then $\psi$ is a solution to
\begin{equation}\label{EL equ}
\n\c(g_\v(|\n\psi|^2,\psi)\n\psi)-\p_zG_\v(|\n\psi|^2,\psi)=0 \q\text{in } \O_{\mu,R}\cap\{\psi<Q\}.
\end{equation}
Furthermore, if $|\n\psi|^2\leq(1-\v)\t_c(\psi)$, it holds that
\begin{equation}\label{dzG_subsonic}
\p_zG_\v(|\n\psi|^2,\psi)=\frac{\B'(\psi)}{g(|\n \psi|^2,\psi)}-\frac{\S'(\psi)}{\g g(|\n \psi|^2,\psi)^\g}.
\end{equation}
\end{lemma}

\begin{proof}
Let $\eta\in C^\infty_0(\O_{\mu,R}\cap\{\psi<Q\})$. Direct computations give
$$\frac{d}{d\vartheta}J_{\mu,R,\Ld}^\v(\psi+\vartheta\eta)\big|_{\vartheta=0}=\int_{\O_{\mu,R}}2\p_tG_\v(|\n\psi|^2,\psi)\n\psi\c\n\eta+\p_zG_\v(|\n\psi|^2,\psi)\eta.$$
By the definition of $G_\v$ in \eqref{G}, minimizers of \eqref{variation problem} satisfy the equation \eqref{EL equ}.

In addition, note that
\begin{equation}\label{gm=g}
	g_\v=g \q \text{ in } \mathcal R_\v:=\{(t,z): 0\leq t\leq(1-\v)\t_c(z),\, z\in\R\}.
\end{equation}
By the definitions of $G_\v$ in \eqref{G} and $g_\v$ in \eqref{g modified}, for $(t,z)\in \mathcal R_\epsilon$, one has
\begin{align*}
\p_zG_\v(t,z)&=\frac12\int_0^t\p_z g(\tau,z)d\tau-(\g-1)g(0,z)^{-\g-1}\p_zg(0,z)\S(z)+\frac{\g-1}{\g}g(0,z)^{-\g}\S'(z)\\
&=: \text{I}+\text{II}+\text{III}.
\end{align*}
Using \eqref{dzg dtg} yields
\begin{align*}
\text{I}&=-\int_0^t\p_t g(\tau,z)\frac{\B'(z)-g(\tau,z)^{1-\g}\S'(z)}{g(\tau,z)^2}d\tau\\
&=\B'(z)\int_0^t\p_t\r(\tau,z)d\tau-\frac1\g\S'(z)\int_0^t\p_t(\r(\tau,z)^\g)d\tau\\
&=\B'(z)\left(\r(t,z)-\r(0,z)\right)-\frac1\g\S'(z)\left(\r(t,z)^\g-\r(0,z)^\g\right).
\end{align*}
Since taking $t=0$ in \eqref{bdz} gives
\begin{equation}\label{dzrho_0z}
\p_z\r(0,z)=\frac{\B'(z)-\r(0,z)^{\g-1}\S'(z)}{(\g-1)\r(0,z)^{\g-2}\S(z)},
\end{equation}
one has
\begin{equation}\label{II}\begin{split}
\text{II}=(\g-1)\r(0,z)^{\g-1}\p_z\r(0,z)\S(z)
=\r(0,z)(\B'(z)-\r(0,z)^{\g-1}\S'(z)).
\end{split}\end{equation}
Moreover,
$$\text{III}=\frac{\g-1}{\g}\r(0,z)^\g\S'(z).$$
Combining the above equalities together yields
\begin{equation}\label{dzG}
	\p_zG_\v(t,z)=\frac{\B'(z)}{g(t,z)}-\frac{\S'(z)}{\g g(t,z)^\g},  \quad (t,z)\in \mathcal R_{\v}.
\end{equation}
This finishes the proof for the lemma.
\end{proof}

The following lemma implies that minimizers for \eqref{variation problem} satisfy the free boundary condition in a weak sense.
The proof is the same as that in \cite[Lemma 3.2]{LSTX2023}, so we omit it here.

\begin{lemma}\label{free BC}
Let $\psi$ be a minimizer for \eqref{variation problem}. Assume that $\O_{\mu,R}\cap\{\psi<Q\}$ is open. Then
$$\Phi_\v\left(|\n\psi|^2,\psi\right)=\ld_\v^2 \quad\text{on } \G_{\psi}:=\O_{\mu,R}\cap\p\{\psi<Q\}$$
in the sense that
$$\lim_{s\to0+}\int_{\p\{\psi<Q-s\}}\left(\Phi_\v\left(|\n\psi|^2,\psi\right)-\ld_\v^2\right)(\eta\c\nu)d\mathcal{H}^1=0\q\text{for any }
\eta\in C_0^\infty(\O_{\mu,R};\R^2),$$
where $\mathcal{H}^1$ is the one-dimensional Hausdorff measure.
\end{lemma}

\begin{remark}
In view of Lemma \ref{free BC}, if a minimizer $\psi$ is smooth near the free boundary $\Gamma_\psi$ and $\Gamma_\psi$ is also smooth, then by the monotonicity of $t\mapsto \Phi_\v(t,z)$ for each $z$ (cf. \eqref{Phi_dt}) and the definition of $\ld_\v$ in \eqref{J} one has $|\n\psi|=\Ld$ on $\G_\psi$.
Moreover, since
	$$\lambda_\epsilon^2=\Phi_\epsilon(\Lambda^2, Q) =\Phi_\epsilon(0, Q)+\int_0^{\Lambda^2}\p_t\Phi_\epsilon(s,Q)ds,$$
it follows from \eqref{Phi_dt} and $\Phi_\epsilon(0,Q)=0$ that
	\begin{align*}
		\frac12c_*(\bar\rho^*)^{-1}\Lambda^2\leq \lambda_\epsilon^2\leq \frac12 (c_*\v\bar\rho_*)^{-1}\Lambda^2,
	\end{align*}
where $c_*=c_*(\g)>0$ is defined in Lemma \ref{g properties},  $\bar\rho_*$ and $\bar\rho^*$ are the lower and upper bounds of $\bar\rho$ defined in \eqref{thm_condition1} and \eqref{rhou_bar_upper}, respectively.
\end{remark}

At the end of the section, we summarize the properties of the function $G_\v$. For this we denote
\begin{equation}\label{label_3}
	\MG(\mp,z):=G_\epsilon(|\mp|^2,z).
\end{equation}

\begin{proposition}\label{Gproperties}
	Let $G_\epsilon$ be defined in \eqref{G} and $\MG$ be defined in \eqref{label_3}. Then $\MG(\mp,z)$ is smooth in $\mp$ and $C^{1,1}$ in $z$. Furthermore, suppose that the density $\bar\rho$ and pressure $\bar P$ of the flow at the upstream satisfy \eqref{thm_condition2} and \eqref{Pbar_*}, where the constant $\bar\k=\bar\k(\bar\rho_*, \bar u,\g,\bar H)$ in \eqref{thm_condition2} is sufficiently small.
	Then the following properties hold:
	\begin{enumerate}
		\item[(i)] {There exist  positive constants  $\mathfrak b_\ast=c_*(\bar\rho^*)^{-1}$ and $\mathfrak b^\ast=(c_*\bar\rho_*)^{-1}$, where  $c_\ast>0$ is a constant depending only on $\gamma$, such that}
		\begin{align}
			\mathfrak b_*|\mp|^2 &\leq p_i \partial_{p_i} \MG (\mp,z) \leq \mathfrak b^*|\mp|^2,\label{eq:convex}\\
			\mathfrak b_* |\mathbf\xi |^2 &\leq \xi_i\partial_{p_ip_j} \MG(\mp,z)\xi_j \leq \mathfrak b^* \epsilon^{-1} |\mathbf\xi|^2 \quad\text{for all }\mathbf\xi\in \R^2.\label{eq:convex0}
		\end{align}
		\item[(ii)] One has
		\begin{equation}\label{supportG}
			\begin{split}
				&\p_z\mcG(\bp,z)= 0 \quad\text{in } \{(\bp, z): \bp\in\R^2, z\in (-\infty,0]\},\\
				&\p_z\mcG(\bp,z)\geq 0 \quad \text{in } \{(\bp, z): \bp\in\R^2, z\in [Q,\infty)\}.
			\end{split}
		\end{equation}
		\item[(iii)] There exist constants
		\begin{align}\label{delta}
			&\delta':=
			C\v^{-1}\k_0 \q \text{and}\q \delta:=\frac{(1+\k_0)\delta'}{\bar\rho_*},
		\end{align}
		where $\kappa_0$ and $\k_1$ are defined in \eqref{k0} and $C=C(\g,\bar u_*)>0$, such that
		\begin{align}
			&\epsilon^{-1}|\p_z \MG(\mp,z)|+|\mp\cdot \p_{\mp z} \MG(\mp,z)|\leq \delta',\label{eq:upper_pzzG1}\\
			&|\p_{\mp z}\MG(\mp,z)|+|\p_{zz}\MG(\mp,z)|\leq \delta, \label{eq:upper_pzzG2}\\
			&\MG(\mathbf 0,Q)=0,\quad \MG(\mp,z)\geq {\frac{\mathfrak b_*}2}|\mp|^2-
			\v\d'{\min\{Q, (Q-z)_+\}}.\label{eq:com_energy}
		\end{align}
	\end{enumerate}
\end{proposition}

\begin{proof}
	\emph{(i).} In view of the definition of $G_\v$ in \eqref{G}, straightforward computations yield
	\begin{align*}
		&p_i\p_{p_i}\MG(\mp,z)=g_\v(|\mp|^2,z)|\mp|^2,\\
		&\p_{p_ip_j}\MG(\mp,z)=g_\v(|\mp|^2,z)\d_{ij}+2\p_tg_\v(|\mp|^2,z)p_ip_j.
	\end{align*}
	Thus from \eqref{gm} and \eqref{g dtgmt} one immediately gets  \eqref{eq:convex} and \eqref{eq:convex0}.
	
	\emph{(ii).} From the definition of $\MG$ and the equality \eqref{gm=g} one has
	\begin{equation}\label{label_9}\begin{split}
			\p_z\MG(\mp,z)&=\frac12\int_0^{|\mp|^2}\p_zg_\v(\tau,z)d\tau-(\g-1)g(0,z)^{-\g-1}\p_zg(0,z)\S(z)+\frac{\g-1}{\g}g(0,z)^{-\g}\S'(z)\\
			&\stackrel{\eqref{II}}{=}\frac12\int_0^{|\mp|^2}\p_zg_\v(\tau,z)d\tau+\frac{\B'(z)}{g(0,z)}-\frac{\S'(z)}{\g g(0,z)^\g}.
	\end{split}\end{equation}
	This, together with the explicit expression for $\p_zg_\epsilon$ in \eqref{gm_dz} and an integration by parts yields that
	\begin{equation}\label{dzGm expression}
		\begin{split}
			\p_z \mcG(\bp,z)=&-\int_0^{|\bp|^2}\left(
			\frac{\mathcal{B}'(z)}{g(\tau,z)}-\frac{\mathcal{S}'(z)}{\g g(\tau,z)^\g}-\frac{1}2
			(g^\ast - g(\tau,z)) \frac{\tau \t_c'(z)}{\t_c(z)} \right)\frac{\varpi'_\epsilon(\tau/\t_c(z))}{\t_c(z)} \ d\tau\\
			&+\left(\frac{\mathcal{B}'(z)}{g(|\bp|^2, z)}-\frac{\mathcal{S}'(z)}{\g g(|\bp|^2, z)^\g}\right)\varpi_\epsilon(|\bp|^2/\t_c(z)).
		\end{split}
	\end{equation}
	Note that $\mathcal B'(z)=\mathcal S'(z)=0$ on $(-\infty, 0]$ and $\B'(z)\geq0$, $\S'(z)=0$ on $[Q,\infty)$, cf. \eqref{eq:sign_B}. If we can show that
		\begin{align}\label{label_8}
		g(\tau,z)(g^\ast - g(\tau,z)) \left(\frac{2\B(z)}{(\g+1)\S(z)}\right)^{\frac2{\g-1}}\leq 1
	\end{align}
	in the set $\mathcal R:=\{(\tau,z):\tau\in [0,\t_c(z)),\, z\in(-\infty,0]\cup[Q,\infty)\}$, then it follows from the expression of $\t_c'$ in \eqref{tc_deri_expression},  $\varpi'_\v\leq 0$, $\varpi_\v\geq 0$ and $g\geq 0$ that  $\p_z \mcG(\bp, z)$ is a positive multiple of $\mathcal{B}'(z)$. In view of \eqref{eq:sign_B}, we consequently conclude that \eqref{supportG} holds true.
    Now it remains to prove \eqref{label_8}. Using \eqref{g bound} and \eqref{rho_bound} yields
	\begin{align}\label{label_12}
		0\leq 4g(\tau,z)(g^\ast-g(\tau,z))\leq (g^\ast)^2 = \left(\frac{2}{\g+1}\right)^{-\frac2{\g-1}}(\bar\rho_*)^{-2}.
	\end{align}
	By \eqref{label_12} and the upper bound of $\B/\S$ in \eqref{B/S}, one can check that the inequality \eqref{label_8} is equivalent to
	\begin{align}\label{rho*_*}
		\frac{\bar\rho^*}{\bar\rho_*}\leq\left(\frac{2^\g}{\g+1}\right)^{\frac1{\g-1}}.
	\end{align}
	Since $\bar P>\bar P_*>\bar\rho_*(\bar u_*)^2/\g$ (cf. \eqref{Pbar_*}), the inequality \eqref{rho*_*} is guaranteed by the condition \eqref{thm_condition2} provided $\bar\k=\bar\k(\bar\rho_*, \bar u,\g,\bar H)$ is sufficiently small.
	This finishes the proof of (ii).
	
	\emph{(iii).} In order to prove \eqref{eq:upper_pzzG1}-\eqref{eq:com_energy}, we let the constant $\bar\k$ in \eqref{thm_condition2} be sufficiently small such that $\bar\k\leq \k_0(\bar\rho_*)^2$. Consequently,
	\begin{equation}\label{label_4}
	\frac{\bar P\k_1}{(\bar\rho_*)^2}\leq \k_0. 
	\end{equation}
	
	\emph{Proof for \eqref{eq:upper_pzzG1}:} It follows from the estimate for $\t_c'$ in \eqref{tc_deri1}, the bounds of $g$ in \eqref{g bound}, the estimate \eqref{rho_bound}, and the $L^\infty$-bounds of $\B'$ and $\S'$ in \eqref{k0_BS} that
	\begin{align}\label{label_11}
		|g^\ast\t_c'(z)|+\left|\frac{\mathcal{B}'(z)}{g(\tau, z)}\right|+\left|\frac{\mathcal{S}'(z)}{\g g(\tau, z)^\g}\right|\leq C C^*,
	\end{align}
	where $C=C(\g,\bar u_*)>0$ and
	\begin{align*}
	C^*:=\frac{\bar\rho^*}{\bar\rho_*}\left(1+\frac{\bar\rho^*}{\bar\rho_*}\right)\left(\k_0+\frac{\bar P\k_1}{(\bar\rho_*)^2}\left(1+\left(\frac{\bar\rho^*}{\bar\rho_*}\right)^{\g-1}\right)\right)\stackrel{\eqref{rho*_*}-\eqref{label_4}}{\leq} C(\gamma)\k_0.
	\end{align*}
	Thus by the expression of $\p_z\MG$ in \eqref{dzGm expression} and the support condition of $\varpi_\epsilon'$,
	one has
	\begin{align*}
		|\p_z\MG(\mp,z)|\leq CC^*
	\end{align*}
	for some $C=C(\g,\bar u_*)>0$.
	Besides, a direct differentiation of $\p_z\MG$ yields
	\begin{align}\label{label_7}
		\p_{p_iz}\MG(\mp,z) = p_i \p_zg_\epsilon (|\mp|^2,z).
	\end{align}
	Thus by \eqref{dzgm2} and \eqref{rho*_*}-\eqref{label_4} one has
	\begin{align*}
		|\mp\cdot \p_{\mp z}\MG(\mp,z)|=|\mp|^2|\p_zg_\epsilon(|\mp|^2,z)|\leq 
		C\v^{-1}\k_0
	\end{align*}
	with $C=C(\g,\bar u_*)$.
	This finishes the proof of \eqref{eq:upper_pzzG1}.
	
	\emph{Proof for \eqref{eq:upper_pzzG2}:} Using \eqref{label_7}, the estimates of $\p_zg_\v$ in \eqref{dzgm1} and \eqref{label_1}, the upper bound of $\t_c$ in \eqref{tc_bound}, and \eqref{rho*_*}-\eqref{label_4}, one gets
	\begin{align}\label{label_10}
		|\p_{\mp z}\MG(\mp,z)|=|\mp||\p_zg_\epsilon(|\mp|^2,z)|
		&\leq C\v^{-1}(\bar P\bar\rho_*)^{-\frac12}\k_0
	\end{align}
	with $C=C(\g,\bar u_*)$.
	
	For the estimate of $\p_{zz}\MG$, note that by the expression of $\p_z\MG$ in \eqref{label_9} and the expression of $\p_z\rho(0,z)$ in \eqref{dzrho_0z} it holds
	\begin{equation}\label{dzzG}\begin{split}
			\p_{zz}\MG(\mp,z)=&\frac12\int_0^{|\mp|^2}\p_{zz}g_\v(\tau,z)d\tau+\frac{\B''(z)}{g(0,z)}-\frac{\S''(z)}{\g g(0,z)^\g}\\
			&+\left(\B'(z)-\frac{\S'(z)}{g(0,z)^{\g-1}}\right)^2\frac{g(0,z)^{\g-2}}{(\g-1)\S(z)}.
	\end{split}\end{equation}
	In view of the estimates about derivatives of $\B$ and $\S$ in \eqref{k0_BS}, the bounds of $g$ in \eqref{g bound}, the lower bound of $\S$ in \eqref{BS_bound}, and \eqref{rho*_*}-\eqref{label_4}, the last three terms on the right-hand side of \eqref{dzzG} are bounded by
	\begin{align*}
		\frac{C\k_0}{\bar\rho_*}\left(1+\k_0+\frac{\k_1}{\bar\rho_*}\right)+\frac{C\k_0^2}{\bar P},
	\end{align*}
	where $C=C(\gamma,\bar u_*)$. To estimate the first term on the right-hand side of \eqref{dzzG}, noting that from \eqref{gm_dz} and that  $\p_z g (t,z)=  2(\B'(z)-g(t,z)^{1-\g}\S'(z))\p_t(\frac{1}{g(t,z)})$, cf. \eqref{dzg dtg}, one has
	\begin{align*}
		\p_{zz}g_\v(\tau,z) =&2\left(\mathcal{B}''(z)\p_t\left(\frac{1}{g(\tau,z)}\right) + \mathcal{B}'(z) \p_{tz} \left(\frac{1}{g(\tau,z)}\right) \right) \varpi_\v\left(\tau/\t_c(z)\right) \\
		&-\frac2{\gamma}\left(\mathcal{S}''(z)\p_t\left(\frac{1}{g(\tau,z)^\gamma}\right) + \mathcal{S}'(z) \p_{tz} \left(\frac{1}{g(\tau,z)^\gamma}\right) \right) \varpi_\v\left(\tau/\t_c(z)\right) \\
		&+  2\p_z g(\tau,z)\frac{d}{dz}\varpi_\v\left(\tau/\t_c(z)\right) +(g(\tau,z)-g^\ast)\frac{d^2}{dz^2}\varpi_\v\left(\tau/\t_c(z)\right).
	\end{align*}
	We plug it into the expression of $\p_{zz}\MG$. For the term involving $\p_{tz}(\frac{1}{g(\tau,z)})$, an integration by parts yields
	\begin{align*}
		&\int_0^{|\mp|^2} \p_{tz}\left(\frac{1}{g(\tau,z)}\right)\varpi_\epsilon\left(\tau/\t_c(z)\right) \ d\tau\\ =& -\frac{\p_zg}{g^2}(|\mp|^2, z)\varpi_\epsilon \left(|\mp|^2/\t_c(z)\right)+\frac{\p_zg}{g^2}(0,z)
		+\int_0^{|\mp|^2}\frac{\p_zg}{g^2}(\tau,z) \frac{\varpi'_\epsilon(\tau/\t_c(z))}{\t_c(z)}\ d\tau.
	\end{align*}
	The term involving $\p_{tz}(\frac1{g^{\gamma}(\tau,z)})$ can be  estimated similarly as above. Note that
	$$|\varpi''_\epsilon(\tau/\t_c(z))|\leq 8\epsilon^{-2}\quad \text{and}\quad |{\rm supp}(\varpi''_\epsilon(\tau/\t_c(z)))|\leq \v.$$
	Thus we infer from the estimates \eqref{k0_BS}, the bounds for $g$ in \eqref{g bound}, the upper bound for $\p_tg$ in \eqref{eq:g_dt_bound}, the relation between $\p_zg$ and $\p_tg$ in \eqref{dzg dtg}, the estimates of $\t_c'$ and $\t_c''$ in \eqref{tc_deri1}-\eqref{tc_deri2} and \eqref{rho*_*}-\eqref{label_4} that
	\begin{equation}\label{label_5}\begin{split}
	|\p_{zz}\MG|\leq&\frac{C\v^{-1}\k_0}{\bar\rho_*}\left(1+\k_0+\frac{\k_1}{\bar\rho_*}\right)+\frac{C\v^{-1}\k_0^2}{\bar P},
		\end{split}
	\end{equation}
	where $C=C(\gamma,\bar u_*)$. Note that $\bar P$ satisfies \eqref{Pbar_*} and $\k_1<\bar\k/\bar P$
	by \eqref{thm_condition2} for $\bar\k\leq \k_0(\bar\rho_*)^2$.
	The inequality \eqref{label_5} together with \eqref{label_10} gives the estimate \eqref{eq:upper_pzzG2}.
	
	\emph{Proof for \eqref{eq:com_energy}:} In view of the definition of $\MG$ in \eqref{label_3} and the definition of $G_\epsilon$ in \eqref{G}, one gets $\MG(\mathbf0,Q)=0$. The second inequality in \eqref{eq:com_energy} follows directly from the strong convexity in $\mp$ and the estimates of derivatives in $z$ in \eqref{eq:upper_pzzG1}. More precisely, the strong convexity property \eqref{eq:convex0} gives that
	\begin{align*}
		\MG(\mp,z)\geq \MG(\mathbf 0,z)+ \frac{\mathfrak b_\ast}{2}|\mp|^2,
	\end{align*}
	where we have also used that $\nabla_{\mathbf p} \MG(\mathbf 0,z)=0$ (since  $\nabla_{\mathbf p} \MG(\mp,z)=2{\mathbf p}\p_tg_\epsilon(|\mp|^2,z)$). To estimate $\MG(\mathbf0,z)$ we use a first order Taylor expansion at $z=Q$.
	{Then for $z\in [0,Q]$, one has
		\begin{align*}
			|\MG(\mathbf0,z)|=|\MG(\mathbf0,Q)+ \p_z\MG(0,\xi_z)(z-Q)|\leq  \v\d' (Q-z),
		\end{align*}
		where $|\p_z\MG(\mp, z)|\leq \v\d'$ from \eqref{eq:upper_pzzG1} is used in the last inequality. Moreover, we infer from 
		\eqref{supportG} that
		$\mathcal G(\mathbf 0,z)=\mathcal G(\mathbf 0,0)\geq - \v\d'Q$ 
		for $z<0$, and $\mathcal G(\mathbf 0,z)\geq \mathcal G(\mathbf 0,Q)=0$ for $z>Q$.}
	Combining the above estimates together we obtain the second inequality in \eqref{eq:com_energy}.
	This completes the proof of the proposition.
\end{proof}

\section{Existence, regularity and fine properties of solutions to the truncated problem}\label{sec:truncated problem}

In the previous sections, we reformulate the truncated free boundary problem into a variational problem \eqref{variation problem}, which has exactly the same formulation as in \cite{LSTX2023}. With Proposition \ref{Gproperties} at hand, one can argue along the same lines as in \cite[Sections 4-6]{LSTX2023} to obtain the existence, regularity and fine properties of a minimizer and its free boundary for the variational problem \eqref{variation problem}. For completeness we briefly states the conclusions here.

\subsection{Existence and regularity of solutions to the truncated problem}
The existence of minimizers for the variational problem \eqref{variation problem} follows from standard theory for calculus of variations.
\begin{lemma}(\cite[Lemma 4.2]{LSTX2023})\label{lem:minimizer_existence}
	Assume $\MG$ defined in \eqref{label_3} satisfies \eqref{eq:convex0} and \eqref{eq:com_energy}. Then the problem \eqref{variation problem} has a minimizer.
\end{lemma}

Throughout the rest of this section, we always assume that the conditions \eqref{thm_condition2} and \eqref{Pbar_*} hold, where the constant $\bar\k=\bar\k(\bar\rho_*, \bar u,\g,\bar H)$ in \eqref{thm_condition2} is sufficiently small. Then the function $\MG$ defined in \eqref{label_3} satisfies all properties in Proposition \ref{Gproperties}.
In view of \cite[Section 4.2]{LSTX2023}, a minimizer $\psi$ has the following properties.

\begin{lemma}
	Let $\psi$ be a minimizer for \eqref{variation problem}. Then the following statements hold:
	\begin{itemize}
		\item[(i)] $\psi$ is a supersolution of the elliptic equation
		\begin{equation}\label{basic_eq_psi}
			\partial_i(\p_{p_i} \MG(\nabla\psi, \psi))- \p_z \MG(\nabla\psi,\psi) = 0,
		\end{equation}
		in the sense of
		\begin{equation}\label{eq:supersol}
			\int_{\O_{\mu,R}}(\p_{p_i} \MG(\nabla\psi, \psi)\partial_i \zeta + \p_z \MG(\nabla\psi,\psi) \zeta)\geq 0,\quad\text{for all }\zeta\geq 0,\,\zeta\in C_0^\infty(\O_{\mu,R}).
		\end{equation}
		\item[(ii)]  $\psi$ satisfies $0 \leq \psi\leq Q$.	
		\item[(iii)] $\psi\in C^{0,\alpha}_{loc}(\O_{\mu,R})$ for any $\alpha\in (0,1)$.
	\end{itemize}
\end{lemma}

Note that by the definition of $Q$ in \eqref{def:Q} and the upper bound of $\bar\rho^*/\bar\rho_*$ in \eqref{rho*_*} one has
\begin{equation}\label{Q_rho}
	\bar\rho_*\|\bar u\|_{L^1([0,\bar H])}\leq Q\leq \bar\rho^*\|\bar u\|_{L^1([0,\bar H])}\leq C\bar\rho_*\|\bar u\|_{L^1([0,\bar H])},
\end{equation}
where the constant $C>0$ only depends on $\g$. Then with the help of the comparison principle for the elliptic equation \eqref{basic_eq_psi} (cf. \cite[Lemma 4.5]{LSTX2023}), one can prove the following (optimal) Lipschitz regularity and the nondegeneracy property of the minimizers as in \cite[Section 4.3]{LSTX2023}.
	\begin{lemma}\label{thm:Lipschitz}
	{Let $\psi$ be a minimizer for \eqref{variation problem}, then $\psi\in C^{0,1}_{loc}(\Omega_{\mu,R})$.} Moreover, for any connected domain $K\Subset \Omega_{\mu,R}$ containing a free boundary point, the Lipschitz constant of $\psi$ in $K$ is estimated by $C\Lambda$, where $C$ depends on $\gamma,\,\epsilon,\,\bar u,\, \frac{\Lambda}{Q},\, K$, and ${\rm dist}(K,\p\Omega_{\mu,R})$.
\end{lemma}

	\begin{lemma}\label{lem:nondeg}
	Let $\psi$ be a minimizer for \eqref{variation problem}. Then for any $p>1$ and any $0<r<1$, there exist {positive constants $c_r=c_r(\gamma,\epsilon,\bar u, r,p)$ small
		and $R_0=R_0(\gamma,\epsilon,\bar u, r)\in (0,1)$, 
		such that for any $B_{R}(\bar x)\subset \Omega_{\mu,R}$ with $R\leq \min\{R_0, c_r\frac{\Lambda}{Q}\}$},
	if
	\begin{equation*}
		\frac{1}{R}\left( 
		\dashint_{B_R(\bar x)}|Q-\psi|^p\right)^{1/p} \leq c_r\Lambda,
	\end{equation*}
	then $\psi=Q$ in $B_{rR}(\bar x)$.
\end{lemma}

The measure theoretic properties of the free boundary follows from the Lipschitz regularity and nondegeneracy of the minimizers.
Then with the help of the estimates for $|\n\psi|$ near the free boundary and the improvement of flatness arguments, one obtains the following regularity of the free boundary.
\begin{proposition}
	Let $\psi$ be a minimizer for \eqref{variation problem}. Let $\mcG$ be the function defined in \eqref{label_3} which  satisfies all properties in Proposition \ref{Gproperties}. Then the free boundary $\Gamma_\psi$ is locally $C^{k+1,\alpha}$ if $\mcG(\bp,z)$ is $C^{k, \alpha}$ in its components ($k\geq1, 0<\alpha<1$), and it is locally real analytic if $\mcG(\bp,z)$ is real analytic.
\end{proposition}
\begin{remark}
Under our assumptions that $\bar\rho,\, \bar u\in C^{1,1}([0,\bar{H}])$, the function $\MG(\mp,z)$ is smooth in $\mp$ and $C^{1,1}$ in $z$. Thus in our case the free boundary is locally $C^{2,\alpha}$ for any $\alpha\in (0,1)$.
\end{remark}

\subsection{Fine properties of the free boundary.}
In order to show the fine properties of the free boundary, we take a specific boundary value on $\p\O_{\mu,R}$ for the variational problem \eqref{variation problem}.
Let $s\in (\frac{1}{2},1)$ be a fixed constant and $b_\mu\in(1,\bar H)$ such that $\Theta(b_\mu)=-\mu$, where $\Theta$ is defined in \eqref{nozzle}. Choose a point $(-\mu, b'_\mu)$ with $1<b'_\mu< b_\mu$ and $k_\mu:=b_\mu-b'_\mu$ small depending on $\bar u, \gamma,$ and $\v$. Define $\psi^\dag(x_2)$ as the minimizer of $\int_{0}^{\tilde{\ubar H}}G_\epsilon(|v'|^2,v)$ in the admissible set $K^\dag:=\{v\in C^{0,1}([0,\tilde{\ubar H}];\mathbb R): v(0)=0, v(\tilde{\ubar H})=Q\}$, where $\tilde{\ubar H}\in(0,1)$ is small enough such that $(\psi^\dag)'(\tilde{\ubar H})>\Lambda$.
Set
\begin{equation}\label{eq:bdry_datum}
	\psi^{\sharp}_{\mu,R}(x_1,x_2):=\left\{
	\begin{aligned}
		&0 && \hbox{if } x_1=-\mu, \,\, 0<x_2<b'_\mu,\\
		&Q\left(\frac{x_2-b'_\mu}{k_\mu}\right)^{1+s} && \hbox{if } x_1=-\mu, \,\, b'_\mu\leq x_2\leq b_\mu,\\
		&\min\{\psi^\dag,Q\} &&\hbox{if } x_1=R, \,\, 0<x_2<1,\\
		&0 &&\hbox{if } x_2=0,\\
		&Q &&\hbox{if } (x_1, x_2)\in \N_1\cup \left([0,R]\times \{1\}\right).
	\end{aligned}
	\right.
\end{equation}
Note that $\psi^{\sharp}_{\mu,R}$ is continuous and it satisfies $0\leq\psi^{\sharp}_{\mu,R}\leq Q$.
If $k_{\mu}$ is sufficiently small,  $\psi^{\sharp}_{\mu,R}(-\mu,\c)$ is a subsolution and $\psi^{\sharp}_{\mu,R}(R,\c)$ is a supersolution to \eqref{EL equ} in $\O_{\mu,R}$.

Let $\psi_{\mu,R,\Ld}$ be a minimizer for \eqref{variation problem} with the boundary value $\psi_{\mu,R}^\sharp$.
By the strong maximum principle, one has
$$\psi^{\sharp}_{\mu,R}(-\mu,x_2)<\psi_{\mu,R,\Ld}(x_1,x_2)\leq\psi^{\sharp}_{\mu,R}(R,x_2)\q \text{ for all } (x_1,x_2)\in \O_{\mu,R}.$$
Then one can show that $\psi_{\mu,R,\Lambda}$ is the unique minimizer for  \eqref{variation problem}, and $\psi_{\mu,R,\Lambda}$ is monotone increasing in the $x_1$-direction, i.e., $\p_{x_1}\psi_{\mu,R,\Ld}\geq0$ in $\O_{\mu,R}$. Consequently, the free boundary $\Gamma_{\psi_{\mu,R,\Lambda}}$ (which is not empty with the boundary value $\psi^{\sharp}_{\mu,R}$) is an $x_2$-graph. More precisely, there exists a function $\ur_{\MH,\Ld}\in(-\mu,R]$ and $\ul H_{\mu,R,\Ld}\in[\tilde{\ubar H},1)$ such that
$$\G_{\psi_{\mu,R,\Lambda}}=\{(x_1,x_2): x_1=\ur_{\mu,R,\Ld}(x_2),\, x_2\in(\ul H_{\mu,R,\Ld},1)\}.$$
By the non-oscillation property (\cite[Lemma 5.4]{LSTX2023}) for the free boundary and the unique continuation property (cf. \cite{MR1809741} and \cite[Lemma 5.5]{LSTX2023}) for the elliptic equation, one concludes that  $\ur_{\mu,R,\Ld}$ is continuous and  $\lim_{x_2\to1-}\ur_{\mu,R,\Ld}(x_2)$ exists.

\subsection{Continuous fit and smooth fit.}
The following proposition shows the existence of a suitable $\Ld>0$ such that the free boundary $\G_{\psi_{\mu,R,\Ld}}$ fits the outlet $A=(0,1)$ of the nozzle continuously.

\begin{proposition}\label{prop:continuous_fit}
	Let $\psi:=\psi_{\mu,R,\Lambda}$ be the minimizer of \eqref{variation problem} with the boundary value $\psi^\sharp_{\mu,R}$ as in \eqref{eq:bdry_datum}. There exists $\Lambda=\Lambda(\mu, R)>0$ such that $\Upsilon_{\mu,R,\Lambda}(1)=0$. Moreover, $\Lambda$ satisfies $C^{-1} Q \leq \Lambda(\mu, R)\leq C Q$ for positive constant $C$ depending on $\gamma$, $\epsilon$, and $\bar u$, 
	but independent of $\mu$ and $R$.
\end{proposition}

The continuous fit implies the following smooth fit.

\begin{proposition}\label{prop:smooth_fit}
		Let $\psi:=\psi_{\mu,R,\Lambda}$ be the minimizer of \eqref{variation problem} with the boundary value $\psi^\sharp_{\mu,R}$ as in \eqref{eq:bdry_datum} and  $\Lambda$ as in Proposition \ref{prop:continuous_fit}.
	Then $\N_1\cup \Gamma_\psi$ is $C^1$ in a neighbourhood of the point $A=(0,1)$, and $\nabla\psi$ is continuous in a $\{\psi<Q\}$-neighborhood of $A$.
\end{proposition}

\section{Existence and uniqueness of the solution to the jet problem}\label{sec:limit_solution}
In this section, we first remove the domain and subsonic truncations, and then study the far fields behavior of the jet flows. Consequently the existence and uniqueness of the solution to Problem \ref{jet problem} are obtained.
The proof is essentially the same as that in \cite[Sections 7 and 8]{LSTX2023}. Here we only give the details for the removal of subsonic truncation.
\subsection{Remove the the domain truncations.}\label{Remove the domain truncation}

Let the nozzle boundary $\N_1$ defined in \eqref{nozzle} satisfy \eqref{nozzleb}.
Given a density $\bar\rho\in C^{1,1}([0,\bar H])$, a horizontal velocity $\bar u\in C^{1,1}([0,\bar H])$ and a pressure $\bar P>0$ of the flow at the upstream, which satisfy  \eqref{thm_condition1}-\eqref{thm_condition2} and \eqref{Pbar_*} with the constant $\bar\k=\bar\k(\bar\rho_*, \bar u,\g,\bar H)$ in \eqref{thm_condition2} sufficiently small.
Let $\psi_{\mu,R,\Lambda}$ be the minimizer of the problem \eqref{variation problem}
with  $\psi^\sharp_{\mu,R}$  defined in \eqref{eq:bdry_datum}. Then for any $\mu_j, R_j\rightarrow \infty$, there is a subsequence (still labelled by $\mu_j$ and $R_j$) such that
$\Lambda_j:=\Lambda(\mu_j, R_j)\rightarrow \Lambda_\infty$ for some $\Lambda_\infty\in (0,\infty)$ and $\psi_{\mu_j, R_j, \Lambda_j}\rightarrow \psi_\infty$ in $C^{0,\alpha}_{loc}(\Omega)$ for any $\alpha\in (0,1)$.
Furthermore, the following statements hold:
\begin{itemize}
	\item [(i)] The function $\psi:=\psi_\infty$ is a local minimizer for the energy functional, i.e., for any $D\Subset \Omega$,  one has  $J^\epsilon(\psi_\infty)\leq J^\epsilon(\varphi)$ for all  $\varphi=\psi_\infty$ on {$\partial D$}, where
	\begin{align*}
		J^\epsilon(\varphi):=\int_{D} G_\epsilon(|\nabla\varphi|^2,\varphi)+\lambda_{\epsilon,\infty}^2 \chi_{\{\varphi<Q\}}\ dx
		\quad\text{and} \quad \lambda_{\epsilon,\infty}:=\sqrt{\Phi_\epsilon(\Lambda_\infty^{{2}},Q)}.
	\end{align*}
	Besides, $\psi$ solves
	\begin{equation}\label{eq_limiting_sol}
		\left\{
		\begin{aligned}
			&\nabla\cdot\left(g_\epsilon(|\nabla \psi|^2,\psi)\nabla \psi\right)-\p_z G_\epsilon (|\nabla \psi|^2,\psi)=0 &&\text{ in } \mathcal{O},\\
			&\psi =0 &&\text{ on }\N_0,\\
			&\psi =Q &&\text{ on } \N_1 \cup \Gamma_{\psi},\\
			&|\nabla \psi| =\Lambda_\infty &&\text{ on } \Gamma_{\psi},
		\end{aligned}
		\right.
	\end{equation}
	where $\mathcal{O}:=\Omega\cap \{\psi<Q\}$ is the flow region, $\Gamma_{\psi}:=\p\{\psi<Q\}\backslash \N_1$ is the free boundary, and  $\p_zG_\epsilon(|\nabla\psi|^2,\psi)$ satisfies \eqref{dzG_subsonic} in the subsonic region $|\nabla\psi|^2\leq (1-\epsilon)\t_c(\psi)$.
	\item [(ii)] The function $\psi$ is in $C^{2,\alpha}(\mathcal{O})\cap C^1(\overline{\mathcal{O}})$, and it satisfies $\p_{x_1}\psi\geq 0$ in $\Omega$.
	\item[(iii)] The free boundary $\Gamma_{\psi}$ is given by the graph $x_1=\Upsilon(x_2)$, where $\Upsilon$ is a $C^{2,\alpha}$ function as long as it is finite.
	\item[(iv)] At the orifice $A=(0,1)$ one has $\lim_{x_2\rightarrow 1-}\Upsilon(x_2)=0$. Furthermore, $\N_1\cup\Gamma_{\psi}$ is $C^1$ around $A$, in particular, $\Theta'(1)=\lim_{x_2\rightarrow 1-}\Upsilon'(x_2)$.
	\item[(v)] There is a constant $\ubar H\in (0,1)$ such that $\Upsilon(x_2)$ is finite if and only if $x_2\in (\ubar H, 1]$, and $\lim_{x_2\rightarrow \ubar H+} \Upsilon(x_2)=\infty$. Furthermore, there exists a constant $a>0$ sufficiently large, such that $\Gamma_{\psi}\cap \{x_1>a\}= \{(x_1, f(x_1)): a<x_1<\infty\}$ for some $C^{2,\alpha}$ function $f$ and $\lim_{x_1\rightarrow \infty}f'(x_1)=0$.
\end{itemize}

\subsection{Remove the subsonic truncation}\label{Remove the subsonic truncation}
In this subsection, we remove the subsonic truncation introduced in Section \ref{subsec_subsonic truncation} provided the upstream pressure of the flow is sufficiently large. 

\begin{proposition}\label{subsonic pro}
Let $\psi:=\psi_\infty$ be a solution obtained in Section \ref{Remove the domain truncation}.
Suppose that $\bar P>\bar P^*$ for some $\bar P^*=\bar P^*(\bar\rho_*,\bar u, \g, \v,\N_1)\geq\tilde P_*$, where $\tilde P_*$ is defined in \eqref{Pbar_*}.
Then
	\begin{equation}\label{subsonic}
		|\n\psi|^2\leq(1-\v)\t_c(\psi)
	\end{equation}
	where $\t_c$ is defined in \eqref{tc}.
\end{proposition}
\begin{proof}
	In the flow region $\mathcal O:=\Omega\cap\{\psi<Q\}$, $\psi$ satisfies the following equation of non-divergence form
	\begin{equation*}\label{eqnondiv}
		a_\epsilon^{ij}(\nabla\psi,\psi) \p_{ij}\psi=F_\epsilon(|\nabla \psi|^2, \psi),
	\end{equation*}
	where the matrix
	$$(a_\epsilon^{ij})=g_\epsilon(|\nabla \psi|^2, \psi)I_2+2\p_t g_\epsilon(|\nabla \psi|^2, \psi)\nabla\psi\otimes \nabla \psi$$
	is symmetric with the eigenvalues
	\begin{align*}
		\beta_{0,\epsilon}  =g_\epsilon(|\nabla \psi|^2, \psi)\quad \text{and}\quad
		\beta_{1,\epsilon} =g_\epsilon(|\nabla \psi|^2, \psi)+2\p_t g_\epsilon(|\nabla \psi|^2, \psi)|\nabla\psi|^2,
	\end{align*}
	and 
	$$F_\epsilon(|\nabla \psi|^2, \psi)=-\p_z g_\epsilon(|\nabla \psi|^2, \psi) |\nabla\psi|^2+\p_zG_\epsilon(|\nabla \psi|^2, \psi).$$
	
	It follows from Lemma \ref{g properties}, \eqref{eq:upper_pzzG1} and \eqref{rho*_*}-\eqref{label_4} that there exist constants $C_1=C_1(\gamma)>1$ and $C_2=C_2(\g,\bar u_*)>0$ such that
	\begin{equation}\label{subsonic_eq}
		1\leq \frac{\beta_{1,\epsilon}}{\beta_{0,\epsilon}}\leq C_1 \epsilon^{-1}
		\quad\text{and}\quad
		\left|\frac{F_\epsilon}{\beta_{0,\epsilon}}\right|\leq C_2\v^{-1}\k_0\bar\rho^*.
	\end{equation}
	Note that $Q\sim\bar\rho^*$ by \eqref{Q_rho}.
	With \eqref{subsonic_eq} at hand, one can use similar arguments as in \cite[Proposition 3]{MR2607929} to get
	\begin{equation}\label{eq:subsonic}
		\|\psi\|_{C^1(\mathcal O)}\leq C\bigg(1+Q+\Big\|\frac{F_\epsilon}{\beta_{0,\epsilon}}\Big\|_{L^{\infty}(\mathbb R\times[0,Q])}\bigg)\leq C\left(1+\left(1+\k_0\right)\bar\rho^*\right)
	\end{equation}
	where $C=C(\bar u, \gamma,\epsilon,\N_1)$. This together with the lower bound of $\t_c$ in \eqref{tc_bound} and \eqref{rho*_*} yields \eqref{subsonic}, provided $\bar P=\bar P(\bar\rho_*, \bar u, \g,\v,\N_1)$ is sufficiently large.
	\end{proof}
\begin{remark}\label{rmk_subsonic_P}
Note that the upstream density $\bar\rho$ and the upstream pressure $\bar P$ of the flow need satisfy \eqref{thm_condition2} with $\bar\k=\bar\k(\bar\rho_*, \bar u, \g,\bar H)>0$, cf. Section \ref{Remove the domain truncation}. For fixed $\v\in(0,1/4)$, if 
\begin{equation}\label{rhobar_condition_v}
\|\bar\rho'\|_{C^{0,1}([0,\bar H])}<\frac{\bar\k}{\bar P^*},
\end{equation}
where $\bar P^*$ is defined in Proposition \ref{subsonic pro}, then \eqref{subsonic} holds as long as the upstream pressure $\bar P$ satisfies 
\begin{equation}\label{pbar_bound}
\bar P^*<\bar P< \hat P:=\frac{\bar\k}{\|\bar\rho'\|_{C^{0,1}([0,\bar H])}}
\end{equation}
(without loss of generality we assume $\|\bar\rho'\|_{C^{0,1}([0,\bar H])}>0$). Therefore, by \eqref{dzG_subsonic} the limiting solution obtained in Section \ref{Remove the domain truncation} is exactly a solution of  \eqref{jet equ}. 
\end{remark}

\subsection{Asymptotic behavior of the jet flow.}\label{Asymptotic behavior}
The solution obtained in Section \ref{Remove the domain truncation} has the following asymptotic behavior as $x_1\to\pm\infty$.

\begin{proposition}\label{prop:asymptotic}
Let the nozzle boundary $\N_1$ defined in \eqref{nozzle} satisfy \eqref{nozzleb}. Given a density $\bar\rho\in C^{1,1}([0,\bar H])$, a horizontal velocity $\bar u\in C^{1,1}([0,\bar H])$, and
a pressure $\bar P>\bar P_*$ for $\bar P_*$ defined in \eqref{Pbar_*}, such that \eqref{thm_condition1}-\eqref{thm_condition2} hold with the constant $\bar\k=\bar\k(\bar\rho_*,\bar u,\g,\bar H)$ sufficiently small. Let $\psi$ be a solution of \eqref{eq_limiting_sol} with $\Lambda:=\Lambda_\infty>0$. Assume that $\psi$ and the free boundary $\Gamma_\psi$ satisfy the properties (ii)-(v) in Section \ref{Remove the domain truncation} and \eqref{subsonic}. Then for any $\a\in(0,1)$, as $x_1\rightarrow -\infty$,
	\begin{equation}\label{eqbarpsi}
		\psi(x_1,x_2)\rightarrow \bar\psi(x_2):=\int_0^{x_2}(\bar\rho\bar u)(s)ds, \quad\text{in } C_{loc}^{2,\alpha}([0,\bar H));
	\end{equation}
	as $x_1\rightarrow \infty$,
	\begin{equation}\label{ubarpsi}
		\psi(x_1,x_2)\to\ubar\psi(x_2):=\int_0^{x_2}(\ubar\rho\ubar u)(s)ds, \quad\text{in }C_{loc}^{2,\alpha}([0,\ubar H)),
	\end{equation}
	where $\ubar\rho\in C^{1,\alpha}([0,\ubar H])$ is the positive downstream density, $\ubar u\in C^{1,\alpha}([0,\ubar H])$ is the positive downstream horizontal velocity, and $\ubar H>0$ is the downstream asymptotic height. Furthermore, the downstream states $(\ubar{\rho},\ubar{u},\ubar H)$ are uniquely determined by $\bar P$, $\bar\rho$, $\bar u$, $\gamma$, $\epsilon$, and $\Lambda$.
\end{proposition}

	\begin{remark}
	Under the assumptions of Proposition \ref{prop:asymptotic}, we refer from Proposition \ref{prop:equiv_sol} and Remark \ref{rmk_p_constant} that the downstream pressure $\ubar P$, which is defined by
	\begin{equation*}
		\ubar P=\frac{(\gamma-1)\S(\ubar\psi)}{\gamma g(|\n\ubar\psi|^2,\ubar\psi)^\gamma},
	\end{equation*}
	must be a constant.
	Thus using \eqref{Spsi} one has
	$$\ubar P=\frac{(\g-1)\ubar\rho^\g(\ubar H)\S(Q)}{\g}.$$
	\end{remark}

From Propositions \ref{prop:asymptotic}, we get the asymptotics for the solution obtained in Section \ref{Remove the domain truncation}. As shown in \cite[Proposition 7.2]{LSTX2023}, one has $\p_{x_1}\psi_\infty>0$ in the flow region. Thus in view of Proposition \ref{subsonic pro}, we have proved the existence of solutions to {Problem} \ref{jet problem} when the upstream pressure $\bar P$ is sufficiently large.
Then the existence of solutions to Problem \ref{problem1} follows from Proposition \ref{prop:equiv_sol}.

\subsection{Uniqueness of the jet flow.}
The following proposition gives the uniqueness of the solution to the jet problem. For the proof we refer to \cite[Proposition 8.1]{LSTX2023}.

\begin{proposition}\label{uniqueness pro}
		Let the nozzle boundary $\N_1$ defined in \eqref{nozzle} satisfy \eqref{nozzleb} and \eqref{nozzle3}.
		Given a density $\bar\rho$, a horizontal velocity $\bar u$, and a pressure $\bar P$ of the flow at the upstream as in Proposition \ref{prop:asymptotic}.
		If $(\psi_i, \Gamma_i, \Lambda_i)$ ($i=1$, $2$) are two uniformly subsonic solutions to Problem \ref{jet problem}, then they must be the same.
\end{proposition}

\section{Existence of the critical pressure}\label{Existence of critical flux}

For given density $\bar\rho$, horizontal velocity $\bar u$, and pressure $\bar P$ of the flow at the upstream as in Section \ref{Remove the domain truncation}, we have shown that there exist subsonic jet  flows as long as $\bar P>\bar{P}^*$ for some $\bar{P}^*$ sufficiently large. In this section, we decrease $\bar{P}^*$ as small as possible until some critical pressure $\bar{P}_c$, such that for $\bar{P}\leq\bar{P}_c$ either the flow fails to be global subsonic or the stream function $\psi$ fails to be a solution in the sense of Problem \ref{jet problem}.

\begin{proposition}
Let the nozzle boundary $\N_1$ defined in \eqref{nozzle} satisfy \eqref{nozzleb} and \eqref{nozzle3}.
Given a density $\bar\rho\in C^{1,1}([0,\bar H])$ and a horizontal velocity $\bar u\in C^{1,1}([0,\bar H])$ satisfying \eqref{thm_condition1}. Assume that $\bar\rho$ satisfies \eqref{rhobar_condition_v}.  
There exists a critical pressure $\bar P_c\in[\bar P_*,\bar P^*]$ such that if $\bar P\in(\bar P_c,\hat P)$, where $\bar P_*$ is defined in \eqref{Pbar_*}, $\bar P^*$ is defined in Proposition \ref{subsonic pro} and $\hat P$ is defined in \eqref{pbar_bound}, then there exists a unique solution $\psi$ to Problem \ref{jet problem}, which satisfies
\begin{equation}\label{Eellipticity}
0< \psi< Q \ \text{ in }\mathcal O\quad\text{and}\quad
\mathfrak M(\bar P):=\sup_{\overline{\mathcal O}}\frac{|\nabla\psi|^2}{\t_c(\psi)}<1,
\end{equation}
where the domain $\mathcal O$ is bounded by $\N_0$, $\N_1$, and $\Gamma$. Furthermore, either $\mathfrak M(\bar P)\rightarrow 1$ as $\bar P\rightarrow \bar P_c+$, or there does not exist a $\sigma>0$ such that Problem \ref{jet problem} has solutions for all $\bar P\in (\bar P_c-\sigma,\bar P_c)$ and
	\begin{equation}\label{Ebifurcationestimate}
		\sup_{\bar P\in(\bar P_c-\sigma,\bar P_c)}\mathfrak M(\bar P)<1.
	\end{equation}
\end{proposition}
\begin{proof}
	The proof is basically the same as in \cite[Proposition 9.1]{LSTX2023}. For completeness we give the details.
	
Let $\{\v_n\}_{n=1}^{\infty}$ be a strictly decreasing sequence of positive numbers such that $\v_n\downarrow0$ and $\v_1\leq\v$ which is used for the subsonic truncations in Section \ref{subsec_subsonic truncation}.
Let $\psi_n(\c;\bar P)$ be the function (if it exists) such that
\begin{itemize}
	\item [(i)] $\psi_n(x; \bar P)\in  C^{2,\alpha}(\{\psi_n(x;\bar P)<Q\})\cap C^{1}(\overline{\{\psi_n(x;\bar P)<Q\}})$ satisfies
\begin{eqnarray}\label{system1}\begin{cases}
		\n\c(g_{\v_n}(|\n\psi|^2,\psi)\n\psi)-\p_zG_{\v_n}(|\n\psi|^2,\psi)=0 &\text{in } \{\psi<Q\},\\
		\psi=0 &\text{on }\N_0,\\
		\psi=Q &\text{on }\N_1\cap \G_\psi,\\
		|\n\psi|=\Ld &\text{on }\G_\psi,
	\end{cases}
\end{eqnarray}
	where the mass flux $Q$ is defined in \eqref{def:Q}, the functions $g_{\epsilon_n}$ and $G_{\epsilon_n}$ are defined in \eqref{g modified} and \eqref{G}, respectively;
	\item [(ii)] $\psi_n(x; \bar P)$ satisfies $\partial_{x_1}\psi_n(x;\bar P)>0$ in $\{\psi_n(x;\bar P)<Q\}$;
	\item [(iii)]  the free boundary $\Gamma_{\psi_n}$ satisfies the properties (1)-(3) in Problem \ref{jet problem}; furthermore, there exist constants $C_1=C_1(\epsilon_n, \bar P, \tau)$ and $C_2=C_2(\epsilon_n, \bar P, a)$ (with $\tau$, $a>0$) depending on $\bar P$ continuously such that
	\begin{equation}\label{eq9.3.5}
		\|\Upsilon\|_{C^{2, \alpha}([\ubar{H}+\tau, 1-\tau ])}\leq C_1(\epsilon_n, \bar P, \tau) \quad \text{and}\quad \|f\|_{C^{2,\alpha}((a, \infty))}\leq C_2(\epsilon_n, \bar P, a),
	\end{equation}
	where $\Upsilon$ and $f$ are the functions appeared in the properties (1) and (3) in Problem \ref{jet problem}, respectively.
\end{itemize}
Moreover, if
$$\frac{|\n\psi_n(x;\bar P)|^2}{\t_c(\psi_n(x;\bar P))}\leq1-\v_n,$$
then $\psi_n(x;\bar P)$ solves \eqref{jet equ}, and it satisfies the far fields behavior claimed in Proposition \ref{prop:asymptotic} and
\begin{equation*}
	0\leq \psi_n(x;\bar P)\leq Q.
\end{equation*}
Set
$$\mathfrak{S}_n(\bar P):=\{\psi_n(x;\bar P):\psi_n(x;\bar P) \text{ satisfies the properties (i)--(iii)}\}.$$
Define
$$\M_n(\bar P):=\inf_{\psi_n\in\mathfrak S_n(\bar P)}\sup_{\overline\MO}\frac{|\n\psi_n(x;\bar P)|^2}{\t_c(\psi_n(x;\bar P))}$$
and
$$\mathfrak{T}_n:=\{s:s\geq \bar P_*,\, \M_n(\bar P)\leq1-\v_n \text{ if } \bar P\in(s,\hat P)\}.$$
It follows from Proposition \ref{subsonic pro} and Remark \ref{rmk_subsonic_P} that $[\bar P^*,\hat P)\subset\mathfrak T_n$ for sufficiently large $n$. Hence $\mathfrak T_n$ is not an empty set. Define $\bar P_n:=\inf\mathfrak T_n$.

The sequence $\{\bar P_n\}$ has some nice properties.

First, $\M_n(\bar P)$ is right continuous for $\bar P\in[\bar P_n,\hat P)$. Let $\{\bar P_{n,k}\}\subset(\bar P_n,\hat P)$ and $\bar P_{n,k}\downarrow \bar P$ as $k\to\infty$. Since $\M_n(\bar P_{n,k})\leq1-\v_n$, one has
$$\|\psi_n(x;\bar P_{n,k})\|_{C^{0,1}(\ov{\MO_{n,k}})}\leq C \q{\rm and}\q \|\psi_n(x;\bar P_{n,k})\|_{C^{2,\a}(\ov{\MO_{n,k}}\backslash B_r(A))}\leq C(r) \text{ for any }r>0,$$
where $\MO_{n,k}$ is the associated domain bounded by $\N_0$, $\N_1$ and the free boundary, and $B_r(A)$ is the disk centered at $A=(0,1)$ with radius $r$. Therefore, there exists a subsequence $\psi_n(x;\bar P_{n,k_m})$ such that $\psi_n(x;\bar P_{n,k_m})\to\psi$ and $\MO_{n,k_m}\to \MO$ as $m\to\infty$. Furthermore, together with the estimates \eqref{eq9.3.5}, $\psi$ satisfies the properties (i)--(iii) mentioned above. Thus $\M_n(\bar P)\leq \lim_{m\to\infty}\M_n(\bar P_{n,k_m})\leq1-\v_n$. It follows from Proposition \ref{uniqueness pro} that $\psi$ is unique. Hence $\M_n(\bar P)=\lim_{m\to\infty}\M_n(\bar P_{n,k_m})$.

Second, $\bar P_n>\bar P_*$. Suppose not, by the definition of $\bar P_n$ one has $\bar P_*\in\mathfrak T_n$. It follows from the right continuity of $\M_n$ that $\M_n(\bar P_*)\leq1-\v_n$. Thus $\psi_n(x;\bar P_*)$ satisfies the far field behavior claimed in Proposition \ref{prop:asymptotic}.
However, it follows from the definition of $\bar P_*$ that
$$\sup_{x\in\overline\MO}\frac{|\n\psi_n(x;\bar P_*)|^2}{\t_c(\psi_n(x;\bar P_*))}\geq \sup_{x_2\in[0,\bar H]}\frac{|(\bar\r\bar u)(x_2;\bar P_*)|^2}{\t_c(\bar\psi(x_2;\bar P_*))}=1.$$
Thus $\M_n(\bar P_*)\geq1$. This leads to a contradiction, which implies $\bar P_n>\bar P_*$.

Finally, it follows from the definition of $\{\bar P_n\}$ that $\{\bar P_n\}$ is an decreasing sequence.

Define $\bar P_c:=\lim_{n\to\infty}\bar P_n$. Based on previous properties of $\{\bar P_n\}$, $\bar P_c$ is well-defined and $\bar P_c\in [\bar P_*,\bar P^*]\subset [\bar P_*,\hat P)$. Note that for any $\bar P\in(\bar P_c,\hat P)$, there exists a $\bar P_n\in(\bar P_c,\bar P)$ and thus $\M_n(\bar P)\leq1-\v_n$. Therefore $\psi=\psi_n(x;\bar P)$ solves Problem \ref{jet problem}.

If $\M(\bar P)\to1$ does not hold as $\bar P\rightarrow \bar P_c+$, then 
$\sup_{\bar P\in(\bar P_c,\hat P)}\M(\bar P)<1$, and hence there exists $n$ such that $\sup_{\bar P\in(\bar P_c,\hat P)}\M(\bar P)<1-\v_n$. As the same as the proof for the right continuity of $\M_n(\bar P)$ on $[\bar P_n,\hat P)$, $\M(\bar P_c)\leq1-\v_n$. If in addition there exists $\sigma>0$ such that Problem \ref{jet problem} always has a solution $\psi$ for $\bar P\in(\bar P_c-\sigma,\bar P_c)$, and
$$\sup_{\bar P\in(\bar P_c-\sigma,\bar P_c)}\M(\bar P)=\sup_{\bar P\in(\bar P_c-\sigma,\bar P_c)}\sup_{\overline\MO}\frac{|\n\psi|^2}{\t_c(\psi)}<1.$$
Then there exists $k>0$ such that
$$\sup_{\bar P\in(\bar P_c-\sigma,\bar P_c)}\M(\bar P)=\sup_{Q\in(\bar P_c-\sigma,\bar P_c)}\sup_{\overline\MO}\frac{|\n\psi|^2}{\t_c(\psi)}\leq 1-\v_{n+k}.$$
This yields that $\bar P_{n+k}\leq \bar P_c-\sigma$ and leads to a contradiction. Therefore, either $\M(\bar P)\to1$ as $\bar P\rightarrow \bar P_c+$, or there does not exist $\sigma>0$ such that Problem \ref{jet problem} has solutions  for all $\bar P\in(\bar P_c-\sigma,\bar P_c)$ and the associated solutions satisfy \eqref{Ebifurcationestimate}.

This completes the proof of the proposition.
\end{proof}

Combining all the results in previous sections, Theorem \ref{result} is proved.

\bibliographystyle{abbrv}
	\bibliography{ref1}

\end{document}